\documentclass[3p]{elsarticle}
\usepackage[T1]{fontenc}
\usepackage{lmodern}
\usepackage{amsmath}
\usepackage{amssymb}
\usepackage{graphicx}
\usepackage{float}      
\usepackage{booktabs}
\usepackage[font=small,labelfont=bf]{caption}
\usepackage[english]{babel}
\usepackage{xcolor}
\usepackage{listings}   
\usepackage[colorlinks=true]{hyperref}
\hypersetup{pdftitle={Fractional-order differential geometry of Gielis curves: chart regularization, Caputo curvature, and the fate of singularities},
            pdfauthor={Umut Selvi, Derya Sağlam}}
\AtBeginDocument{\hypersetup{linkcolor=blue!50!black,citecolor=blue!50!black,urlcolor=blue!50!black}}
\biboptions{sort&compress}

\makeatletter
\def\ps@pprintTitle{%
  \let\@oddhead\@empty
  \let\@evenhead\@empty
  \def\@oddfoot{\reset@font\hfil\thepage\hfil}%
  \let\@evenfoot\@oddfoot}
\makeatother

\newtheorem{theorem}{Theorem}
\newtheorem{proposition}{Proposition}
\newtheorem{corollary}{Corollary}
\newtheorem{lemma}{Lemma}

\newdefinition{example}{Example}
\newdefinition{definition}{Definition}
\newdefinition{remark}{Remark}
\newproof{proof}{Proof}

\newcommand{\R}{\mathbb{R}}
\newcommand{\G}{\mathcal{G}}
\newcommand{\bt}{\mathbf{t}}
\newcommand{\bn}{\mathbf{n}}
\newcommand{\bb}{\mathbf{b}}

\begin{document}

\title{Fractional-order differential geometry of Gielis curves: chart regularization, Caputo curvature, and the fate of singularities}

\author[a]{Umut Selvi}
\author[a]{Derya Sağlam}
\address[a]{Department of Mathematics, Ankara Hacı Bayram Veli University, Ankara 06900, Türkiye}

\begin{abstract}
Gielis curves describe many natural shapes, and some of them have corners or cusps. We apply two recent ideas from the differential geometry of curves to Gielis curves: ``regularizing'' changes of metric and fractional curvatures based on the Caputo derivative. First, we show that the smoothness of a Gielis curve is controlled by the exponents $n_2$ and $n_3$, not by $n_1$. A change of coordinates, the Gielis chart, maps every Gielis curve that closes after one turn onto the unit circle. However, the corners and cusps of the curve reappear as a loss of regularity of the metric; they are transported, not removed. Second, for the fractional curvature defined by Caputo scaling, we distinguish two conventions that differ in the choice of parameter and derive their scaling laws under homotheties. We prove that the memory of the Caputo operator breaks the rotational symmetry: the fractional curvature takes different values at corresponding points of congruent arcs. For fractional order $\alpha<1$, the curves of constant fractional curvature are spirals, and the clothoid is a special case. We also introduce a fractional total turning for curves with corners, give closed forms for regular polygons, and show that fractional general and slant helices coincide with the classical ones. The Python library used in the computations is listed in the appendix.
\end{abstract}


\maketitle
\suppressfloats[t]

\section{Introduction}\label{sec:intro}
The superformula of Gielis~\cite{Gielis2003} extends the superellipses of Lam\'e~\cite{Lame1818} to shapes with any rotational symmetry. With a few parameters, the resulting Gielis curves describe a wide range of natural and abstract shapes, such as leaves, flowers, fruits, starfish and regular polygons~\cite{Gielis2003,Gielis2012,Gielis2017,ShiRatkowskyGielis2020,Li2022}. They are used to fit biological shapes to data~\cite{Shi2018,ShiRatkowskyGielis2020,ShiBiogeom2022}, to solve boundary value problems in domains bounded by Gielis curves~\cite{Caratelli2013}, and to model solids and design antennas~\cite{Fougerolle2005,Fougerolle2006,Simeoni2011}.

The geometry of Gielis curves and surfaces has been studied as well. The names Gielis curves and Gielis surfaces go back to the geometric work of Gielis, Haesen and Verstraelen~\cite{GielisHaesenVerstraelen2005}. Koiso and Palmer~\cite{KoisoPalmer2008} constructed surfaces of constant anisotropic mean curvature for energies whose Wulff shape is given by the Gielis formula. Matsuura~\cite{Matsuura2015} analyzed the curvature of Gielis curves at their corners and at the midpoints of their sides, and compared their limit curves with regular polygons. Anisotropic curvature and Gielis transformations of spirals are studied in~\cite{GielisSpirals2020,Spichal2025}. Recently, a superelliptic framework for the differential geometry of Gielis transformations was proposed~\cite{Ozdemir2026}, and it was extended with quaternions and rotations~\cite{ParlakOzdemir2026,OzdemirGielis2026}.

Fractional derivatives are widely used to model processes with memory. Among them, the Caputo derivative is a common choice, since it annihilates constants~\cite{Caputo1967,Podlubny1999,Kilbas2006,Diethelm2010}. Fractional derivatives have also been used to extend the differential geometry of curves and surfaces~\cite{YajimaYamasaki2012,Lazopoulos2016}. For curves, several approaches exist. One uses a simplified Caputo derivative of a composite function, which rescales the classical Frenet apparatus~\cite{Yajima2018,Aydin2021,HasYilmaz2023}. Another uses the conformable derivative~\cite{Gozutok2019,Has2022}. A third is based on the exact Caputo derivative and leads to a different notion of fractional curvature~\cite{RubioLopez2023,RubioLopez2024,Ogrenmis2025,RubioLopez2026}.

Gielis curves are a natural test case for these ideas. Some of them have corners or cusps, and the superelliptic framework of~\cite{Ozdemir2026} treats all Gielis curves in a uniform way by means of a modified inner product and a modified derivative. This raises the question of what a change of metric can do at a corner or a cusp. Gielis curves also have rotational symmetry, which makes it easy to see how a fractional curvature depends on the base point of the Caputo operator.

In this paper, we first show that the smoothness of a Gielis curve is governed by the exponents $n_2$ and $n_3$, and not by $n_1$ (Lemma~\ref{lem:regular} and Corollary~\ref{cor:geom}). We then study the Gielis chart, which maps every Gielis curve that closes after one turn onto the unit circle. However, the corners and cusps of the curve reappear as a loss of regularity of the metric (Proposition~\ref{prop:chart} and Corollary~\ref{cor:transport}). We also show that the star derivative of the superelliptic framework is not tangent to the transformed curve (Remark~\ref{rem:superelliptic}).

For the fractional curvature defined by Caputo scaling, we distinguish two conventions, which we call E and F. They have different scaling laws under homotheties (Theorem~\ref{thm:scaling} and Corollary~\ref{cor:param}). We prove that the memory of the Caputo operator breaks the rotational symmetry (Theorem~\ref{thm:memory}) and that the curves of constant fractional curvature include the clothoid (Theorem~\ref{thm:cesaro}). We also define a fractional total turning for curves with corners, give conditions for its finiteness and derive closed forms for regular polygons (Definition~\ref{def:theta} and Propositions~\ref{prop:theta} and~\ref{prop:polygon}). Finally, we show that fractional general helices and slant helices coincide with the classical ones~\cite{Struik1988,Izumiya2004,Kula2010} (Proposition~\ref{prop:helix} and Remark~\ref{rem:HY34}). All numerical values in the paper are verified numerically (Section~\ref{sec:numerics}), and the Python library used in the computations is listed in \ref{app:code}.

The paper is organized as follows. Section~\ref{sec:prelim} collects basic facts on Gielis curves and on Caputo scaling. Section~\ref{sec:chart} studies the Gielis chart, and Section~\ref{sec:frac} the fractional curvature of Gielis curves. Section~\ref{sec:turning} treats the fractional total turning of cornered curves, and Section~\ref{sec:helix} fractional helices over Gielis bases. Section~\ref{sec:numerics} describes the computations, and Section~\ref{sec:discussion} discusses the results and open problems.

\section{Preliminaries}\label{sec:prelim}
\subsection{Gielis curves: symmetry, closure and regularity}\label{sec:gielis}
Gielis curves are given in polar form $\gamma(\varphi)=r(\varphi)(\cos\varphi,\sin\varphi)$ by the superformula of Gielis \cite{Gielis2003},
\begin{equation}\label{eq:gielis}
r(\varphi)=\Bigl(\bigl|\tfrac{1}{a}\cos\tfrac{m\varphi}{4}\bigr|^{n_2}+\bigl|\tfrac{1}{b}\sin\tfrac{m\varphi}{4}\bigr|^{n_3}\Bigr)^{-1/n_1},
\qquad a,b,n_2,n_3>0,\ n_1\neq0.
\end{equation}
Throughout, $m$ is a positive integer and we use the exponent notation $(m;n_1,n_2,n_3)$ with $a=b=1$ unless stated otherwise.

\begin{lemma}\label{lem:periodic}
(a) $r$ is even and $4\pi/m$-periodic.
(b) If $a=b$ and $n_2=n_3$, then $r$ is $2\pi/m$-periodic.
(c) Consequently $r$ is $2\pi$-periodic, i.e.\ $\gamma$ closes after one turn, whenever $m$ is even, or $a=b$ and $n_2=n_3$. The curve is then invariant under the reflection in the $x$-axis and under the rotation by $2\pi/K$, where $K=m/2$ in general and $K=m$ in case~(b); it consists of $K$ congruent arcs. If $m$ is odd and (b) fails, then in general $r(\varphi+2\pi)\neq r(\varphi)$ and $\gamma$ closes only after $\varphi$ has run over $[0,4\pi]$.
\end{lemma}
\begin{proof}
(a) $\varphi\mapsto-\varphi$ leaves $|\cos\frac{m\varphi}{4}|$ and $|\sin\frac{m\varphi}{4}|$ unchanged, and $\varphi\mapsto\varphi+4\pi/m$ adds $\pi$ to $m\varphi/4$, which only changes signs. (b) $\varphi\mapsto\varphi+2\pi/m$ adds $\pi/2$ to $m\varphi/4$ and interchanges $|\cos|$ and $|\sin|$. (c) $2\pi=\frac m2\cdot\frac{4\pi}{m}=m\cdot\frac{2\pi}{m}$; evenness of $r$ is the reflection symmetry.
\qed
\end{proof}
For instance $(6;6,4,4)$ has $K=6$ congruent arcs, whereas $(6;6,4,4)$ with $b=1.3$, or $(6;6,4,5)$, has only $K=3$; the curve $(5;7,7,4)$ does not close after one turn. The test curves of this paper are shown in Figure~\ref{fig:gallery}.

\begin{figure}[H]\centering
\includegraphics[width=\textwidth]{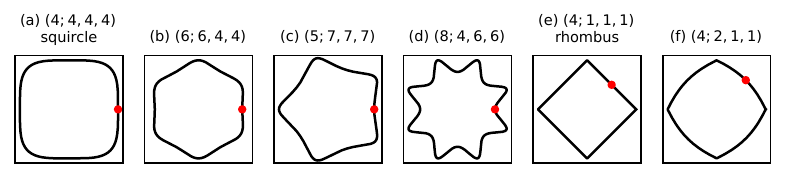}
\caption{Test curves in exponent notation $(m;n_1,n_2,n_3)$, $a=b=1$. Red dot: base point, $\varphi_0=0$, except for the cornered curves $(4;1,1,1)$ and $(4;2,1,1)$, whose point $\varphi=0$ is a vertex; there $\varphi_0=\pi/4$ (an edge midpoint).}\label{fig:gallery}
\end{figure}

\begin{lemma}\label{lem:regular}
Away from the zeros of $\sin(m\varphi/2)$ the function $r$ is real-analytic. Let $\varphi_*$ be a zero of $\sin(m\varphi/4)$ and put $n=n_3$, or a zero of $\cos(m\varphi/4)$ and put $n=n_2$. Near $\varphi_*$, $r$ is of class $C^k$ if and only if $|\varphi-\varphi_*|^n$ is; that is, $r$ is $C^\infty$ (indeed real-analytic) near $\varphi_*$ if $n$ is an even integer, and otherwise $r\in C^k$ near $\varphi_*$ exactly for the integers $k<n$. In particular, the exponent $n_1$ has no influence on the regularity of $r$.
\end{lemma}
\begin{proof}
Let $A$ denote the bracket in \eqref{eq:gielis}. The two terms of $A$ do not vanish simultaneously, so $A>0$, and $y\mapsto y^{-1/n_1}$ is an analytic diffeomorphism of $(0,\infty)$; hence $r$ and $A$ have the same regularity everywhere, whatever $n_1$. Away from the zeros of $\cos\frac{m\varphi}{4}\sin\frac{m\varphi}{4}=\frac12\sin\frac{m\varphi}{2}$ both terms of $A$ are analytic. Near a zero $\varphi_*$ of $\sin\frac{m\varphi}{4}$ the first term is analytic, and $\sin\frac{m\varphi}{4}=(\varphi-\varphi_*)g(\varphi)$ with $g$ analytic and $g(\varphi_*)=\pm m/4\neq0$, so $|\frac1b\sin\frac{m\varphi}{4}|^{n_3}=|\varphi-\varphi_*|^{n_3}h(\varphi)$ with $h$ analytic and positive; multiplication by $h$ preserves the class $C^k$. The same argument applies at the zeros of $\cos\frac{m\varphi}{4}$. Finally, for $x\neq0$ the $k$-th derivative of $|x|^n$ is $n(n-1)\cdots(n-k+1)|x|^{n-k}(\operatorname{sgn}x)^k$. If $n$ is an even integer, $|x|^n=x^n$ is analytic, and so are $A$ and $r$. Otherwise, for $k<n$ this derivative extends continuously to $x=0$ (with value $0$), whereas for the smallest integer $k\ge n$ it is unbounded near $0$ ($n\notin\mathbb Z$) or has a jump at $0$ ($n$ odd); hence $|x|^n\in C^k$ exactly for $k<n$.
\qed
\end{proof}

For $r\in C^2$ we use
\begin{equation}\label{eq:kappa}
\nu=\sqrt{r^2+r'^2},\qquad
\kappa=\frac{r^2+2r'^2-r\,r''}{(r^2+r'^2)^{3/2}},\qquad
u(\varphi)=\int_{\varphi_0}^{\varphi}\nu\,d\varphi',
\end{equation}
where $\kappa$ is the signed curvature (counterclockwise orientation), $u$ is the Euclidean arc length measured from the base point $\varphi_0$, and $\vartheta$ denotes the tangent angle, $d\vartheta/du=\kappa$.

\begin{corollary}\label{cor:geom}
Near a point $\varphi_*$ as in Lemma~\ref{lem:regular}, with governing exponent $n$:
if $n\ge2$, then $r\in C^2$ and $\kappa$ is continuous;
if $1<n<2$, the curve is $C^1$ but $\kappa$ is unbounded;
if $n=1$, $r'$ has a jump and $\gamma$ has a corner;
if $0<n<1$, $|r'|\to\infty$ and $\gamma$ has a cusp (the unit tangent tends to the radial direction, with opposite orientations on the two sides).
\end{corollary}
\begin{proof}
Throughout the proof $x=\varphi-\varphi_*$, and the Landau symbols refer to the limit $x\to0$: $O(1)$ denotes a quantity that remains bounded and $o(1)$ a quantity that tends to $0$; $f=O(|x|^q)$ means $|f|\le M|x|^q$ for some constant $M$ and all small $x\neq0$. As in the proof of Lemma~\ref{lem:regular}, $A=A_0(\varphi)+|x|^nh(\varphi)$ with $A_0,h$ analytic near $\varphi_*$, $h>0$ and $A_0'(\varphi_*)=0$. Since $r=A^{-1/n_1}$, we have for $x\neq0$
\[
r'=-\tfrac{1}{n_1}A^{-1/n_1-1}A',\qquad
r''=-\tfrac{1}{n_1}A^{-1/n_1-1}A''+\tfrac{1}{n_1}\bigl(\tfrac{1}{n_1}+1\bigr)A^{-1/n_1-2}(A')^2 .
\]

If $n\ge2$, Lemma~\ref{lem:regular} gives $r\in C^2$, so $\kappa$ in \eqref{eq:kappa} is continuous.

Now let $1<n<2$. By Lemma~\ref{lem:regular}, $r\in C^1$; in particular $r$ and $r'$ are bounded, and $A'=A_0'+n\operatorname{sgn}(x)|x|^{n-1}h+|x|^nh'=O(1)$. Differentiating once more,
\[
A''=n(n-1)h(\varphi_*)|x|^{n-2}+O(1),
\]
where $O(1)$ collects $A_0''$, $2n\operatorname{sgn}(x)|x|^{n-1}h'$, $|x|^nh''$ and $n(n-1)|x|^{n-2}\bigl(h(\varphi)-h(\varphi_*)\bigr)=O(|x|^{n-1})$, all bounded because $n>1$. The second term in $r''$ is $O(1)$ as well, since $A'=O(1)$ and $A$ is bounded away from $0$. Hence
\[
r''=-\tfrac{1}{n_1}A^{-1/n_1-1}\,n(n-1)h(\varphi_*)\,|x|^{n-2}+O(1).
\]
The first term tends to $\pm\infty$ because $n-2<0$ and $n(n-1)h(\varphi_*)\neq0$; adding a bounded quantity cannot prevent this, so $|r''|\to\infty$. As $r$ and $r'$ are bounded and $r^2+r'^2\ge r^2>0$, $\kappa$ in \eqref{eq:kappa} is unbounded.

If $n=1$, then $A'=A_0'+\operatorname{sgn}(x)h+|x|h'$, and $A_0'(\varphi_*)=0$ gives the one-sided limits $A'(\varphi_*^\pm)=\pm h(\varphi_*)\neq0$. Thus $r'$ has a jump at $\varphi_*$, and so has the direction of $\gamma'=r'\,\mathbf e_r+r\,\mathbf e_\varphi$ (polar frame $\mathbf e_r,\mathbf e_\varphi$): $\gamma$ has a corner.

Finally, let $0<n<1$. We write
\[
A'=n\operatorname{sgn}(x)h(\varphi_*)|x|^{n-1}\bigl(1+o(1)\bigr).
\]
Indeed, the remaining terms $A_0'=O(|x|)$ (because $A_0'(\varphi_*)=0$), $n\operatorname{sgn}(x)|x|^{n-1}\bigl(h(\varphi)-h(\varphi_*)\bigr)=O(|x|^n)$ and $|x|^nh'=O(|x|^n)$ are, relative to $|x|^{n-1}$, of order $O(|x|^{2-n})$, $O(|x|)$ and $O(|x|)$; since $2-n>1$, the relative error is $O(|x|)=o(1)$. Consequently
\[
r'=-\tfrac{1}{n_1}A^{-1/n_1-1}\,n\operatorname{sgn}(x)h(\varphi_*)|x|^{n-1}\bigl(1+o(1)\bigr).
\]
For $|x|$ small the factor $1+o(1)$ is positive, so $\operatorname{sgn}(r')=-\operatorname{sgn}(n_1)\operatorname{sgn}(x)$, and $|r'|\to\infty$ because $n-1<0$. (A bounded error term $O(1)$ in place of $o(1)$ would not suffice here: it could reverse the sign of the factor $1+O(1)$.) Therefore
\[
\frac{\gamma'}{|\gamma'|}=\frac{r'\,\mathbf e_r+r\,\mathbf e_\varphi}{\sqrt{r'^2+r^2}}\longrightarrow\operatorname{sgn}(r')\,\mathbf e_r(\varphi_*)=\mp\operatorname{sgn}(n_1)\,\mathbf e_r(\varphi_*)\qquad\text{as }x\to0^\pm ,
\]
so the unit tangent reverses its direction at $\varphi_*$: $\gamma$ has a cusp.
\qed
\end{proof}

\begin{remark}\label{rem:n1}
In the frequently used diagonal case $n_1=n_2=n_3=n$, the condition $n<2$ does produce singular points; by Lemma~\ref{lem:regular} this is due to $n_2$ and $n_3$, not to $n_1$. For example, $(4;1,4,4)$ is real-analytic although $n_1=1$, whereas $(4;4,1,1)$ has corners although $n_1=4$ (both are verified numerically; see Section~\ref{sec:numerics}). This corrects the heuristic, stated for instance in \cite[Abstract and Section~1]{Ozdemir2026}, that small values of $n_1$ ($n_1<2$) create corner-like singularities. All smooth test curves of this paper have $n_2,n_3\in\{4,6,7\}$, hence $r\in C^6$.
\end{remark}

Figure~\ref{fig:regularity} illustrates Lemma~\ref{lem:regular}, Corollary~\ref{cor:geom} and Remark~\ref{rem:n1}.

\begin{figure}[H]\centering
\includegraphics[width=\textwidth]{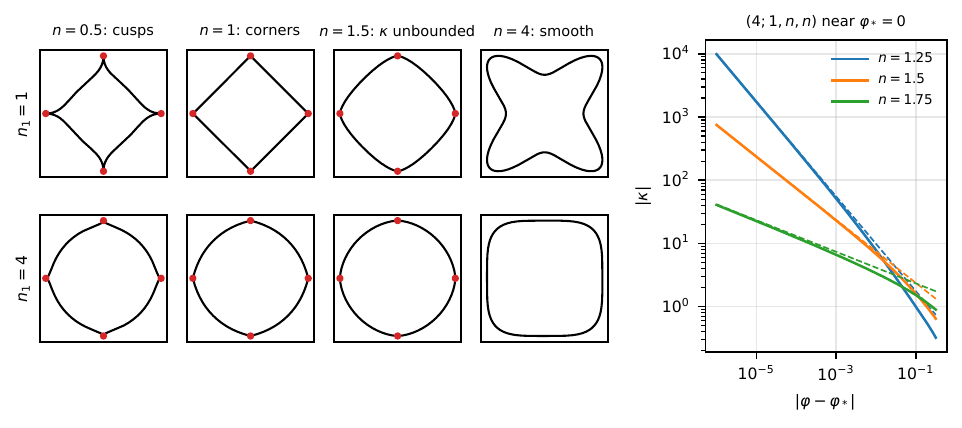}
\caption{Regularity is governed by $n_2=n_3=n$, not by $n_1$. Left: the curves $(4;n_1,n,n)$ for $n_1=1$ (top) and $n_1=4$ (bottom); red dots (drawn for $n<2$): the points $\varphi=k\pi/2$, where the governing exponent is $n$. Each column shows one type of point, in accordance with Corollary~\ref{cor:geom}: cusps ($n<1$), corners ($n=1$), unbounded curvature ($1<n<2$), smooth points ($n\ge2$). Right: $|\kappa|$ near $\varphi_*=0$ for $(4;1,n,n)$ against $|\varphi-\varphi_*|$; dashed: $n(n-1)|\varphi-\varphi_*|^{n-2}$, the leading term of $\kappa$ for $n_1=1$.}\label{fig:regularity}
\end{figure}

\subsection{Caputo scaling: two conventions}\label{sec:caputo}
The Caputo derivative \cite{Caputo1967,Diethelm2010} of order $0<\alpha<1$ is $({}^{C}D^\alpha_s f)(s)=\frac{1}{\Gamma(1-\alpha)}\int_0^s (s-v)^{-\alpha}f'(v)\,dv$; it annihilates constants, and ${}^{C}D^1_s=d/ds$. Put
\[
\lambda_\alpha(s)=\frac{\alpha\,s^{1-\alpha}}{\Gamma(2-\alpha)},\qquad s\ge0 .
\]
Let $x$ be a regular curve with Euclidean arc length $u$ (from the base point) and Frenet frame $\{\bt,\bn,\bb\}$ (for space curves we assume $C^3$ and $\kappa>0$; for plane curves $C^2$ suffices), and let $s=s(u)$ be an increasing $C^1$ parameter with $s(0)=0$. Under the simplification of the Caputo derivative of a composite function introduced by Yajima et al.\ \cite{Yajima2018}, the fractional derivative of $x(s)$ takes the form \cite[eqs.~(2.4), (2.11)]{HasYilmaz2023}
\begin{equation}\label{eq:yajima}
D^\alpha_s x:=\lambda_\alpha(s)\,\frac{dx}{ds}.
\end{equation}
Since $D^\alpha_s x$ is parallel to $dx/ds$, the unit fractional tangent coincides with $\bt$; this is why the Frenet frame does not change \cite[Conclusion~2.8]{HasYilmaz2023}. (The exact Caputo derivative is in general not tangent to the curve; see Section~\ref{sec:discussion}.)

\begin{definition}\label{def:fraccurv}
The fractional curvature and torsion along the parameter $s$ are
\[
\kappa_\alpha[s]=\frac{\bigl\langle d\bt/ds,\bn\bigr\rangle}{|D^\alpha_s x|},\qquad
\tau_\alpha[s]=\frac{\bigl\langle d\bn/ds,\bb\bigr\rangle}{|D^\alpha_s x|}
\]
(for plane curves $\bn$ is $\bt$ rotated by $+\pi/2$ and $\kappa_\alpha$ is signed).
\end{definition}

\begin{lemma}\label{lem:param}
For the quantities of Definition~\ref{def:fraccurv}, $\kappa_\alpha[s]=\kappa(u)/\lambda_\alpha(s(u))$ and $\tau_\alpha[s]=\tau(u)/\lambda_\alpha(s(u))$.
\end{lemma}
\begin{proof}
$d\bt/ds=\kappa\bn\,du/ds$, $d\bn/ds=(-\kappa\bt+\tau\bb)\,du/ds$ and $|D^\alpha_s x|=\lambda_\alpha(s)\,du/ds$.
\qed
\end{proof}
Thus the fractional curvature depends on the parameter only through $\lambda_\alpha(s)$. Two choices are natural; Remark~\ref{rem:HY} explains how both arise from \cite{HasYilmaz2023}. In the first, which we call convention~E (Euclidean arc length), $s=u$, so that
\[
w_E(u):=\lambda_\alpha(u)=C_E\,u^{p_E},\qquad C_E=\frac{\alpha}{\Gamma(2-\alpha)},\quad p_E=1-\alpha .
\]
In the second, convention~F (fractional arc length), $s=s_F$ with unit fractional speed, $|D^\alpha_{s_F}x|\equiv1$, i.e.\ $ds_F/du=\lambda_\alpha(s_F)$. The solution with $s_F(0)=0$ and $s_F(u)>0$ for $u>0$ is
\begin{equation}\label{eq:sF}
s_F(u)=\Bigl(\frac{\alpha^2u}{\Gamma(2-\alpha)}\Bigr)^{1/\alpha},
\end{equation}
which is \cite[eq.~(2.8)]{HasYilmaz2023}; with this parameter \eqref{eq:yajima} is their (2.11) and Lemma~\ref{lem:param} reproduces their (2.13)--(2.14), taken from \cite{Aydin2021}. Here
\[
w_F(u):=\lambda_\alpha(s_F(u))=\frac{ds_F}{du}=C_F\,u^{p_F},\qquad C_F=\alpha^{2/\alpha-1}\,\Gamma(2-\alpha)^{-1/\alpha},\quad p_F=\frac{1-\alpha}{\alpha}.
\]
In both conventions
\begin{equation}\label{eq:scaling}
\kappa_\alpha=\frac{\kappa}{w},\qquad \tau_\alpha=\frac{\tau}{w},\qquad w(u)=C\,u^{p},
\end{equation}
with $C=1$, $p=0$ exactly when $\alpha=1$; for $\alpha<1$, $p_F-p_E=(1-\alpha)^2/\alpha>0$. Table~\ref{tab:conv} lists the constants for some values of $\alpha$, and Figure~\ref{fig:weights} shows the weights.

\begin{table}[H]\centering
\caption{Constants of the weight $w(u)=Cu^p$ in conventions E and F, and the homothety exponent $1+p$ of $\kappa_\alpha$ (Theorem~\ref{thm:scaling}).}\label{tab:conv}
\begin{tabular}{rrrrrrr}
\toprule
$\alpha$ & $C_E$ & $p_E$ & $1+p_E$ & $C_F$ & $p_F$ & $1+p_F$ \\
\midrule
0.25 & 0.2720 & 0.7500 & 1.7500 & $8.555\times10^{-5}$ & 3.0000 & 4.0000 \\
0.5 & 0.5642 & 0.5000 & 1.5000 & 0.1592 & 1.0000 & 2.0000 \\
0.75 & 0.8274 & 0.2500 & 1.2500 & 0.7058 & 0.3333 & 1.3333 \\
1.0 & 1.0000 & 0.0000 & 1.0000 & 1.0000 & 0.0000 & 1.0000 \\
\bottomrule
\end{tabular}
\end{table}

\begin{figure}[H]\centering
\includegraphics[width=\textwidth]{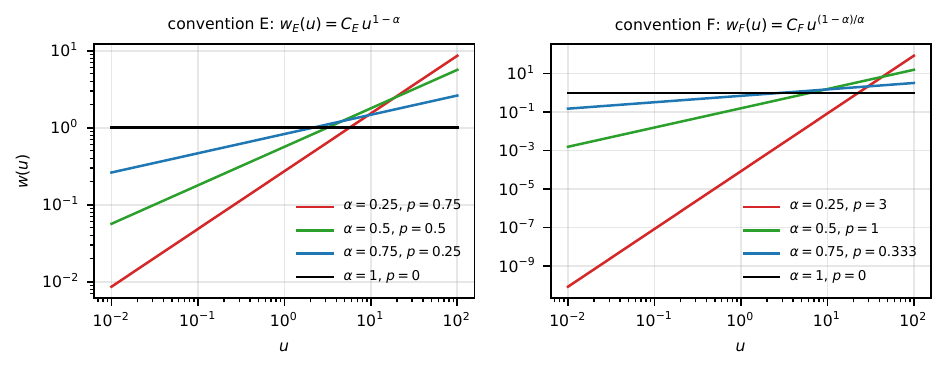}
\caption{The weights $w(u)=C\,u^p$ of the two conventions on logarithmic axes, convention E (left) and F (right), for $\alpha=0.25$, $0.5$, $0.75$ and $1$ (Table~\ref{tab:conv}). The slopes are $p_E=1-\alpha$ and $p_F=(1-\alpha)/\alpha$; for $\alpha<1$ the weight of convention F grows faster and is smaller near $u=0$.}\label{fig:weights}
\end{figure}

\begin{remark}\label{rem:HY}
In \cite{HasYilmaz2023} the parameter $s$ is introduced by (2.8), so it is $s_F$ and not the Euclidean arc length; the phrase ``where $s$ measures its arc length'' in their Theorems~2.6--2.7 has to be understood in the sense of unit fractional speed, $\|d^\alpha x/ds^\alpha\|=1$, established after their (2.11). Reading $s$ as Euclidean arc length in (2.13)--(2.14) yields convention~E instead. The two conventions differ by the factor $w_F/w_E=(C_F/C_E)\,u^{(1-\alpha)^2/\alpha}$; for $\alpha=1/2$, for example, $\kappa^E_\alpha=\kappa/(0.5642\,u^{1/2})$ whereas $\kappa^F_\alpha=2\pi\kappa/u$. We state all results for both conventions.
\end{remark}

\begin{lemma}[Fractional Frenet equations]\label{lem:frenet}
Let $P(u)=\int_0^u w(v)\,dv=C\,u^{p+1}/(p+1)$. With $'=d/dP$ the classical frame satisfies
\begin{equation}\label{eq:frenetP}
\bt'=\kappa_\alpha\bn,\qquad \bn'=-\kappa_\alpha\bt+\tau_\alpha\bb,\qquad \bb'=-\tau_\alpha\bn .
\end{equation}
In convention F, $P=s_F$ of \eqref{eq:sF} and \eqref{eq:frenetP} is \cite[eq.~(2.12)]{HasYilmaz2023}; in convention E, $P_E(u)=\alpha u^{2-\alpha}/\Gamma(3-\alpha)$.
\end{lemma}
\begin{proof}
$d/dP=w^{-1}d/du$, the classical Frenet equations and \eqref{eq:scaling}.
\qed
\end{proof}
In either convention the fractional Frenet apparatus is therefore the classical one reparametrized by $P$. In particular $\kappa_\alpha=|d\bt/dP|$ has the dimension of $1/P$, that is, of $[\text{length}]^{-(1+p)}$.

\begin{remark}\label{rem:base}
Both conventions require a base point ($u=0$), which is part of the data; this is a genuine feature of Caputo (memory) operators. For $\alpha<1$, $w(u)\to0$ as $u\to0^+$, so $|\kappa_\alpha(u)|\to\infty$ at the base point unless $\kappa$ vanishes there.
\end{remark}

\section{The Gielis chart: what ``regularization'' means}\label{sec:chart}
In this section $r$ is a continuous, positive and $2\pi$-periodic radial function (Lemma~\ref{lem:periodic}); $\theta(x)=\arg x\in\R/2\pi\mathbb Z$, so that $r\circ\theta$ is well defined and continuous on $\R^2\setminus\{0\}$.
\begin{definition}[Gielis chart]\label{def:chart}
\[
\Phi:\R^2\setminus\{0\}\to\R^2\setminus\{0\},\qquad \Phi(x)=\frac{x}{r(\theta(x))},\qquad g=\Phi^*\delta,
\]
where $\delta$ is the Euclidean metric.
\end{definition}

\begin{proposition}\label{prop:chart}
(a) $\Phi$ is a homeomorphism with inverse $y\mapsto r(\theta(y))\,y$, and it maps the Gielis curve $\mathcal C=\{x:|x|=r(\theta(x))\}$ onto the unit circle.
(b) For $k\ge1$: $\Phi\in C^k$ if and only if $r\in C^k$. In polar coordinates $(\rho,\theta)$,
\[
g=r^{-2}\,d\rho^2-2\rho\,r'r^{-3}\,d\rho\,d\theta+\rho^2\bigl(r^{-2}+r'^2r^{-4}\bigr)\,d\theta^2 ,
\]
so, if $r\in C^k$, the coefficients of $g$ are $C^{k-1}$, and they are not $C^k$ unless $r\in C^{k+1}$.
(c) If $r\in C^2$, then $\Phi$ is a $C^2$ diffeomorphism and an isometry from $(\R^2\setminus\{0\},g)$ onto $(\R^2\setminus\{0\},\delta)$; hence $\mathcal C$ has $g$-length $2\pi$ and $g$-geodesic curvature identically $1$. If $r\in C^3$, then $g\in C^2$ and its Gaussian curvature vanishes.
\end{proposition}
\begin{proof}
(a) Since $r>0$, $\theta(\Phi(x))=\theta(x)$ and $|\Phi(x)|=|x|/r(\theta(x))$; hence $y\mapsto r(\theta(y))y$ is a two-sided inverse, both maps are continuous, and $|\Phi|=1$ exactly on $\mathcal C$.
(b) In polar coordinates $\Phi(\rho,\theta)=(R,\theta)$ with $R=\rho/r(\theta)$, so the regularity of $\Phi$ is that of $1/r$, i.e.\ of $r$. Inserting $dR=d\rho/r-\rho r'd\theta/r^2$ into $\delta=dR^2+R^2d\theta^2$ gives the formula. Its coefficients involve $r$ and $r'$, hence are $C^{k-1}$; the mixed coefficient $g_{\rho\theta}=-\rho r'r^{-3}=\frac{\rho}{2}(r^{-2})'$ is $C^k$ only if $r^{-2}\in C^{k+1}$, i.e.\ $r\in C^{k+1}$.
(c) By construction; isometries preserve length, geodesic curvature and Gaussian curvature.
\qed
\end{proof}

\begin{corollary}\label{cor:transport}
If the governing exponent of a zero $\theta_*$ (Lemma~\ref{lem:regular}) equals $1$ --- for example for the rhombus $(4;1,1,1)$, or any $(m;n_1,1,1)$ --- then $r'$ jumps at $\theta_*$ (Corollary~\ref{cor:geom}); $\Phi$ is Lipschitz but not $C^1$ across the ray $\theta=\theta_*$, and $g$ has a jump discontinuity across it. If the governing exponent is smaller than $1$ (a cusp), $r'$ is unbounded near $\theta_*$ and $\Phi$ is not even locally Lipschitz. In these cases --- the ones the regularization is meant for --- the geodesic curvature of $\mathcal C$ with respect to $g$ is not defined in the classical sense; whenever it is defined ($r\in C^2$) it equals $1$, but this is encoded in $\Phi$: the non-smoothness of $\mathcal C$ has been moved into the metric.
\end{corollary}

Figure~\ref{fig:chart} illustrates Proposition~\ref{prop:chart} and Corollary~\ref{cor:transport} for the rhombus.

\begin{figure}[H]\centering
\includegraphics[width=\textwidth]{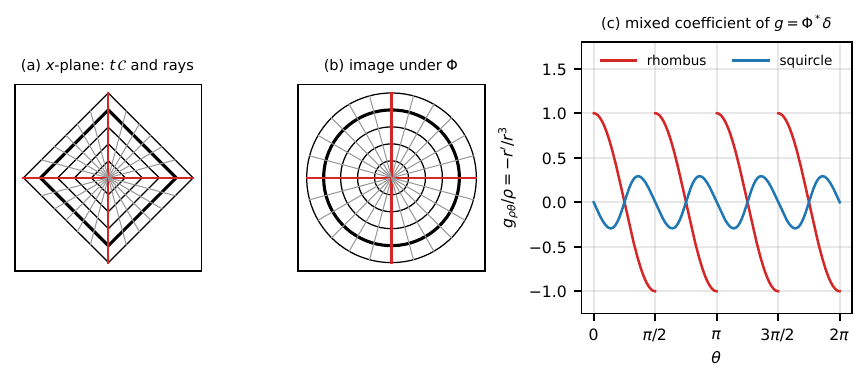}
\caption{The Gielis chart $\Phi(x)=x/r(\theta(x))$ for the rhombus $(4;1,1,1)$ (Definition~\ref{def:chart}, Proposition~\ref{prop:chart}). (a) The curve $\mathcal C$ (thick), its dilates $t\,\mathcal C$ for $t=0.25$, $0.5$, $0.75$, $1.25$, and rays $\theta=\mathrm{const}$ every $15^\circ$; red: the rays through the corners. (b) Their images under $\Phi$: concentric circles and the same rays. (c) The mixed coefficient $g_{\rho\theta}/\rho=-r'/r^3$ of $g=\Phi^*\delta$: for the rhombus it jumps across the corner rays $\theta=k\pi/2$ (Corollary~\ref{cor:transport}), for the squircle $(4;4,4,4)$ it is continuous.}\label{fig:chart}
\end{figure}

\begin{remark}\label{rem:superelliptic}
We compare with \cite{Ozdemir2026} on two points.

(i) The superelliptic inner product $\langle X,Y\rangle_{SE}=r(t)^{-2}\sum_i\sigma_iX_iY_i$ \cite[eq.~(7)]{Ozdemir2026} depends on the parameter $t$ of a curve; it defines a Riemannian metric on the ambient space only if $t$ is expressed through the point, for instance $t=\theta(x)$ in the plane. With this reading and $\sigma_i=1$ it is the conformal metric $\hat g=r(\theta(x))^{-2}\,\delta$ on $\R^2\setminus\{0\}$. For $r\in C^1$ the $\hat g$-length of $\mathcal C$ is $\int_0^{2\pi}\sqrt{1+(r'/r)^2}\,d\varphi$, which exceeds $2\pi$ unless $r$ is constant (numerically by $3.0\%$, $9.4\%$ and $38.6\%$ for $(4;4,4,4)$, $(5;7,7,7)$ and $(8;4,6,6)$). Hence $\mathcal C$ is not a unit circle for $\hat g$; a metric in which $\mathcal C$ is a unit circle is $g=\Phi^*\delta$ of Definition~\ref{def:chart} and Proposition~\ref{prop:chart}, which is not conformal ($g_{\rho\theta}\neq0$ where $r'\neq0$).

(ii) The star derivative $D^*=\G\circ D$ \cite[eq.~(23)]{Ozdemir2026} is not the derivative of the transformed curve. For the unit circle $\mathbf e(t)=(\cos t,\sin t)$ and $\sigma_i=1$, the transformed curve $\G(\mathbf e)(t)=r(t)\mathbf e(t)$ has velocity $r'(t)\mathbf e(t)+r(t)\mathbf e'(t)$, whereas $\G(\mathbf e')(t)=r(t)\mathbf e'(t)$; the two differ by $r'(t)\mathbf e(t)$, so $\G(\mathbf e')$ is not even tangent to $\G(\mathbf e)$ where $r'\neq0$ (for $(6;6,4,4)$ the angle reaches $19.5^\circ$, Figure~\ref{fig:star}). For a general curve $\mathbf c$ (denoted by another letter in \cite{Ozdemir2026}), the relation $\mathbf c^*_{SE}=\G(\mathbf c')$ of \cite[Proposition~1]{Ozdemir2026}, obtained ``by differentiating'', drops the corresponding term, and the constant-curvature statements of \cite{Ozdemir2026} (e.g.\ their Figure~5) concern $\G(\mathbf c')$ rather than the geometry of the curve $\G(\mathbf c)$. In \cite[Remark~2]{Ozdemir2026} the order ``differentiate first, then scale'' is motivated by interpreting $r$ as a metric radius; our objection is that the resulting vector field is not the velocity of any parametrization of $\G(\mathbf c)$.

This is why we base the present work on the chart $\Phi$, a genuine change of coordinates, and compute all differential invariants in the ordinary Euclidean sense.
\end{remark}

\begin{figure}[H]\centering
\includegraphics[width=\textwidth]{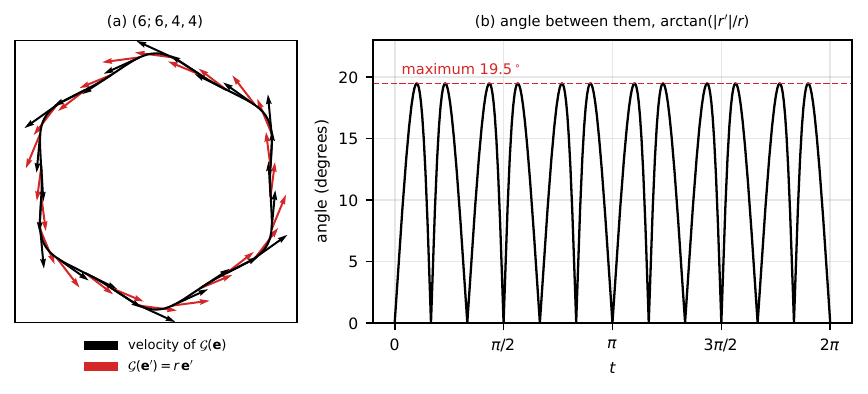}
\caption{Remark~\ref{rem:superelliptic}(ii) for $(6;6,4,4)$ and $\sigma_i=1$. (a) At 24 points: the velocity $r'\mathbf e+r\mathbf e'$ of the transformed curve $\G(\mathbf e)$ (black) and the vector $\G(\mathbf e')=r\mathbf e'$ (red); they differ by $r'\mathbf e$. (b) The angle between them, $\arctan(|r'|/r)$: it vanishes where $r'=0$ and reaches $19.5^\circ$.}\label{fig:star}
\end{figure}

\section{Fractional curvature of Gielis curves}\label{sec:frac}
For a closed Gielis curve with $r\in C^2$ and base point $\varphi_0$ we write $\kappa_\alpha(\varphi)=\kappa(\varphi)/w(u(\varphi))$, with $\kappa$, $u$ from \eqref{eq:kappa} and $w\in\{w_E,w_F\}$; superscripts E, F indicate the convention when it matters.

\begin{theorem}[Scaling law]\label{thm:scaling}
Let $\mathcal C$ be a regular $C^2$ curve in $\R^2$ or $\R^3$ ($C^3$ with $\kappa>0$ for the statement on $\tau_\alpha$), $c>0$, and let $c\,\mathcal C$ be its image under the homothety $x\mapsto cx$, the base point being mapped to the base point. At corresponding points
\[
\kappa_\alpha\longmapsto c^{-(1+p)}\kappa_\alpha,\qquad \tau_\alpha\longmapsto c^{-(1+p)}\tau_\alpha,
\]
that is, with exponent $2-\alpha$ in convention E and $1/\alpha$ in convention F.
\end{theorem}
\begin{proof}
$\kappa\mapsto\kappa/c$, $\tau\mapsto\tau/c$, $u\mapsto cu$, and $w(cu)=c^p\,w(u)$.
\qed
\end{proof}

\begin{corollary}\label{cor:param}
Let $n_2=n_3=n$. Replacing $(a,b)$ by $(ca,cb)$ multiplies $r$ by $\hat c=c^{n/n_1}$; it is the homothety of ratio $\hat c$, and $\kappa_\alpha\mapsto\hat c^{-(1+p)}\kappa_\alpha$. If $n_2\neq n_3$ and $c\neq1$, the ratio of the new and the old radial function equals $c^{n_2/n_1}$ at the zeros of $\sin(m\varphi/4)$ and $c^{n_3/n_1}$ at the zeros of $\cos(m\varphi/4)$; hence the parameter change is not a homothety centred at the origin and Theorem~\ref{thm:scaling} does not apply.
\end{corollary}
\begin{proof}
Write the bracket in \eqref{eq:gielis} as $A=A_{\cos}+A_{\sin}$ with $A_{\cos}=|\cos\frac{m\varphi}{4}/a|^{n_2}$ and $A_{\sin}=|\sin\frac{m\varphi}{4}/b|^{n_3}$. The new radial function is $(c^{-n_2}A_{\cos}+c^{-n_3}A_{\sin})^{-1/n_1}$; evaluate where $A_{\sin}=0$, resp.\ $A_{\cos}=0$, and use Theorem~\ref{thm:scaling}.
\qed
\end{proof}
Table~\ref{tab:scaling} confirms Corollary~\ref{cor:param} numerically, and Figure~\ref{fig:homothety} shows both of its cases. Note that $c^{-(2-\alpha)}$, i.e.\ the convention-E exponent applied to $c$ instead of $\hat c$, is correct only if $\hat c=c$, i.e.\ $n_2=n_3=n_1$; for $(6;6,4,4)$ and $\alpha=1/2$ it would give $0.3536$ instead of $0.5$.

\begin{table}[H]\centering
\caption{Check of Corollary~\ref{cor:param} with $c=2$: ratio of $\kappa_\alpha$ for $(ca,cb)$ and $(a,b)$ at corresponding points versus $\hat c^{-(1+p)}$, $\hat c=c^{n/n_1}$; median over the grid points with $|\kappa|>10^{-2}\max|\kappa|$, where all pointwise deviations are below $10^{-5}$. Closer to the zeros of $\kappa$ the pointwise ratio is affected by rounding errors of the spectral second derivative.}\label{tab:scaling}
\begin{tabular}{lcc rr rr}
\toprule
 & & & \multicolumn{2}{c}{convention E} & \multicolumn{2}{c}{convention F}\\
\cmidrule(lr){4-5}\cmidrule(lr){6-7}
curve & $\hat c=c^{n/n_1}$ & $\alpha$ & measured & $\hat c^{-(2-\alpha)}$ & measured & $\hat c^{-1/\alpha}$ \\
\midrule
$(4;4,4,4)$ & 2.0000 & 0.5 & 0.353553 & 0.353553 & 0.250000 & 0.250000 \\
$(4;4,4,4)$ & 2.0000 & 0.8 & 0.435275 & 0.435275 & 0.420448 & 0.420448 \\
$(6;6,4,4)$ & 1.5874 & 0.5 & 0.500000 & 0.500000 & 0.396850 & 0.396850 \\
$(6;6,4,4)$ & 1.5874 & 0.8 & 0.574349 & 0.574349 & 0.561231 & 0.561231 \\
$(8;4,6,6)$ & 2.8284 & 0.5 & 0.210224 & 0.210224 & 0.125000 & 0.125000 \\
$(8;4,6,6)$ & 2.8284 & 0.8 & 0.287175 & 0.287175 & 0.272627 & 0.272627 \\
\bottomrule
\end{tabular}
\end{table}

\begin{figure}[H]\centering
\includegraphics[width=\textwidth]{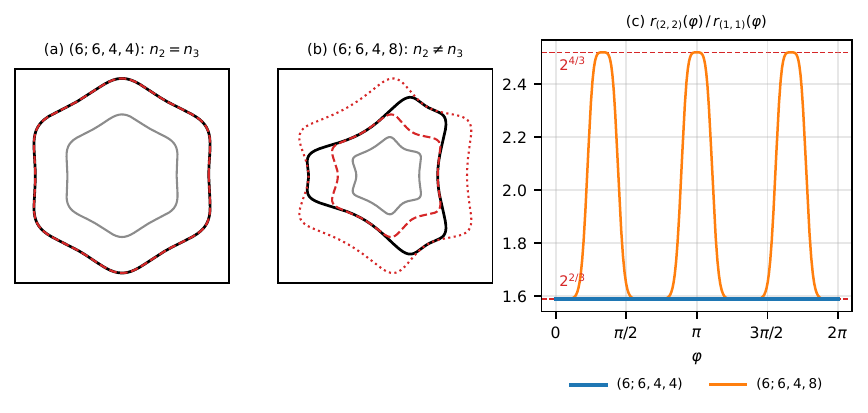}
\caption{Corollary~\ref{cor:param} with $c=2$. (a) $(6;6,4,4)$, $n_2=n_3$: the curves with $(a,b)=(1,1)$ (grey) and $(2,2)$ (black) are homothetic; the grey curve scaled by $\hat c=2^{2/3}$ (dashed) coincides with the black one. (b) $(6;6,4,8)$, $n_2\neq n_3$: the grey curve scaled by $2^{2/3}$ (dashed) lies inside the black one and scaled by $2^{4/3}$ (dotted) outside it; neither coincides with it. (c) The ratio of the two radial functions: constant $2^{2/3}$ for $(6;6,4,4)$, between $2^{2/3}$ and $2^{4/3}$ for $(6;6,4,8)$.}\label{fig:homothety}
\end{figure}

\begin{theorem}\label{thm:memory}
Let $r\in C^2$ be $2\pi/K$-periodic with $K\ge2$ (Lemma~\ref{lem:periodic}), let $L$ be the length of the curve and $\ell=L/K$. Then $\kappa(j\ell+\xi)=\kappa(\xi)$ and
\[
\kappa_\alpha(j\ell+\xi)=\frac{\kappa(\xi)}{w(j\ell+\xi)},\qquad j=0,\dots,K-1,\ \ \xi\in(0,\ell).
\]
If $w$ is strictly increasing, which is the case for $0<\alpha<1$ in both conventions, and $\kappa(\xi)\neq0$, then $|\kappa_\alpha(j\ell+\xi)|$ is strictly decreasing in $j$. Since $\kappa\not\equiv0$, $\kappa_\alpha$ is not invariant under $u\mapsto u+j\ell$ for any $j\in\{1,\dots,K-1\}$: none of the rotational symmetries of the curve survives.
\end{theorem}
\begin{proof}
The rotation by $2\pi/K$ maps $\gamma(\varphi)$ to $\gamma(\varphi+2\pi/K)$ and preserves arc length and signed curvature; since $\nu$ is $2\pi/K$-periodic, $u(\varphi+2\pi/K)=u(\varphi)+\ell$. Thus $\kappa$ is $\ell$-periodic in $u$, and $w$ is strictly increasing because $p>0$ for $\alpha<1$.
\qed
\end{proof}

\begin{remark}\label{rem:basepeak}
By Remark~\ref{rem:base}, $|\kappa_\alpha|$ is unbounded near the base point whenever $\kappa(\varphi_0)\neq0$; for $(6;6,4,4)$, $\kappa(0)=-1/2$ and $\kappa_\alpha\to-\infty$ as $u\to0^+$. Comparisons between arcs must therefore be made at corresponding points, as in Theorem~\ref{thm:memory}, not via suprema over arcs.
\end{remark}

For $(6;6,4,4)$ we have $a=b$ and $n_2=n_3$, hence $K=m=6$ congruent arcs of length $\ell=1.1379$. Each lobe is symmetric under a reflection, so $\kappa$ is symmetric about the arc midpoints $u_j=(2j+1)\ell/2$, and the computation shows that the lobe tips (maxima of $\kappa$) lie there; Theorem~\ref{thm:memory} therefore predicts the ratios $(2j+1)^{-p}$. Figure~\ref{fig:kappa_alpha} and Table~\ref{tab:peaks} confirm this: for $\alpha=1/2$ the second lobe tip retains $1/\sqrt3\approx58\%$ (convention E) and $1/3\approx33\%$ (convention F) of the value at the first. This is a feature of the Caputo operator (non-locality with a distinguished starting point), not of the curve. A periodic Caputo operator restarted at each symmetry point would restore the symmetry; we leave its construction open. Figure~\ref{fig:combs} shows the decay of $\kappa_\alpha$ from lobe to lobe along the whole curve.

\begin{figure}[H]\centering
\includegraphics[width=\textwidth]{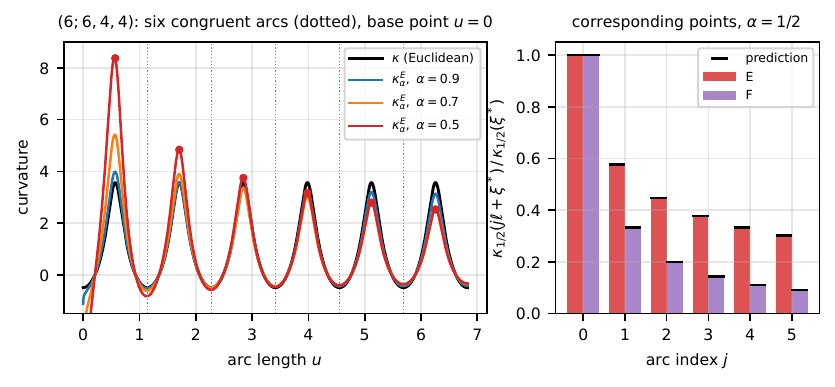}
\caption{Left: Euclidean curvature $\kappa$ and $\kappa^E_\alpha$ ($\alpha=0.9$, $0.7$, $0.5$) of the $(6;6,4,4)$ curve; dotted lines separate the $K=6$ congruent arcs; since $\kappa(0)=-1/2$, $\kappa^E_\alpha\to-\infty$ as $u\to0^+$ (clipped; Remark~\ref{rem:basepeak}). Right: $\kappa_{1/2}$ at the lobe tips $u_j=j\ell+\xi^*$, normalized by $j=0$, in both conventions; black marks: Theorem~\ref{thm:memory}.}\label{fig:kappa_alpha}
\end{figure}

\begin{table}[H]\centering
\caption{Fractional curvature at corresponding points (lobe tips, $\xi^*=\ell/2$) of the $K=6$ congruent arcs of $(6;6,4,4)$, $\alpha=1/2$, base point $\varphi_0=0$. Ratios are relative to $j=0$; in parentheses the prediction $(\xi^*/(j\ell+\xi^*))^{p}=(2j+1)^{-p}$ of Theorem~\ref{thm:memory}.}\label{tab:peaks}
\begin{tabular}{rrrrrrr}
\toprule
$j$ & $u_j=j\ell+\xi^*$ & $\kappa(u_j)$ & $\kappa^E_{1/2}(u_j)$ & ratio (Thm.~\ref{thm:memory}) & $\kappa^F_{1/2}(u_j)$ & ratio (Thm.~\ref{thm:memory}) \\
\midrule
0 & 0.5689 & 3.563595 & 8.3740 & 1.0000 (1.0000) & 39.3555 & 1.0000 (1.0000) \\
1 & 1.7068 & 3.563595 & 4.8347 & 0.5774 (0.5774) & 13.1185 & 0.3333 (0.3333) \\
2 & 2.8447 & 3.563595 & 3.7450 & 0.4472 (0.4472) & 7.8711 & 0.2000 (0.2000) \\
3 & 3.9825 & 3.563595 & 3.1651 & 0.3780 (0.3780) & 5.6222 & 0.1429 (0.1429) \\
4 & 5.1204 & 3.563595 & 2.7913 & 0.3333 (0.3333) & 4.3728 & 0.1111 (0.1111) \\
5 & 6.2583 & 3.563595 & 2.5249 & 0.3015 (0.3015) & 3.5778 & 0.0909 (0.0909) \\
\bottomrule
\end{tabular}
\end{table}

\begin{figure}[H]\centering
\includegraphics[width=\textwidth]{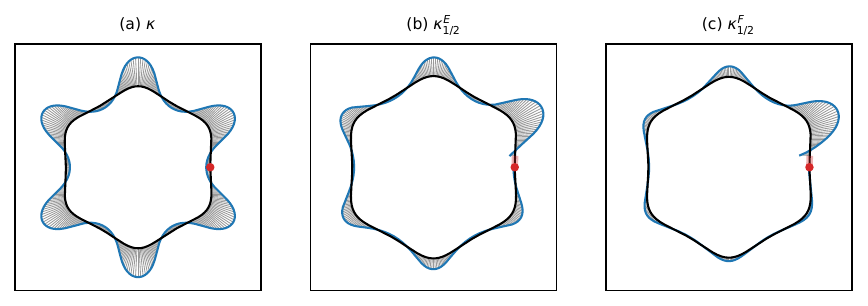}
\caption{Curvature combs of $(6;6,4,4)$ with base point $\varphi_0=0$ (red dot): teeth along the normal, of length proportional to (a) $\kappa$, (b) $\kappa^E_{1/2}$, (c) $\kappa^F_{1/2}$, pointing outwards where the value is positive; each panel has its own scale. The Euclidean comb has the six-fold symmetry of the curve, whereas the fractional combs decrease from lobe to lobe (Theorem~\ref{thm:memory}, Table~\ref{tab:peaks}). On the arc $u<\ell/8$ next to the base point (light red), where $\kappa_\alpha\to-\infty$ (Remark~\ref{rem:basepeak}), no teeth are drawn; the envelope starts inside the curve because $\kappa_\alpha<0$ there.}\label{fig:combs}
\end{figure}

\begin{theorem}\label{thm:cesaro}
(a) The fractional curvature of a circle of radius $R$ is $\kappa_\alpha(u)=1/(R\,w(u))$, which is not constant for $\alpha<1$.
(b) A planar $C^2$ curve has constant fractional curvature $\kappa_\alpha\equiv\kappa_0>0$ if and only if its tangent angle is $\vartheta(u)=\vartheta(0)+\kappa_0P(u)$, i.e.
\[
\text{E:}\ \ \vartheta(u)=\vartheta(0)+\frac{\kappa_0\,\alpha}{\Gamma(3-\alpha)}\,u^{2-\alpha},\qquad
\text{F:}\ \ \vartheta(u)=\vartheta(0)+\kappa_0\,s_F(u)=\vartheta(0)+\kappa_0\Bigl(\frac{\alpha^2u}{\Gamma(2-\alpha)}\Bigr)^{1/\alpha},
\]
and such a curve is unique up to rigid motion. For $\alpha=1$ both are circles. In convention F with $\alpha=1/2$, $\kappa(u)=\kappa_0u/(2\pi)$: the curve is exactly the clothoid (Euler--Cornu spiral). In convention E, with $\kappa_0=c/\alpha$ and $\alpha\to0$, $\vartheta\to cu^2/2$, so the clothoid appears only as a formal limit.
\end{theorem}
\begin{proof}
(a) is \eqref{eq:scaling} with $\kappa=1/R$. (b) $\kappa_\alpha\equiv\kappa_0$ if and only if $\vartheta'=\kappa=\kappa_0w$, i.e.\ $\vartheta=\vartheta(0)+\kappa_0P$; use $\Gamma(3-\alpha)=(2-\alpha)\Gamma(2-\alpha)$, $P_F=s_F$, and $C_F=\frac18\,\Gamma(\frac32)^{-2}=\frac{1}{2\pi}$, $p_F=1$ for $\alpha=\frac12$. A planar curve is determined by $\kappa(u)$ up to rigid motion.
\qed
\end{proof}
In convention F the condition $\kappa_\alpha\equiv\kappa_0$ means that the tangent angle is linear in the fractional arc length, $\vartheta=\vartheta(0)+\kappa_0s_F$ --- the exact analogue of $\vartheta=\vartheta(0)+\kappa_0u$ for circles. Figure~\ref{fig:cesaro} shows these ``fractional circles''; for $\alpha<1$, $\kappa=\kappa_0w(u)$ increases with $u$, so they are spirals that wind inwards.

\begin{figure}[H]\centering
\includegraphics[width=\textwidth]{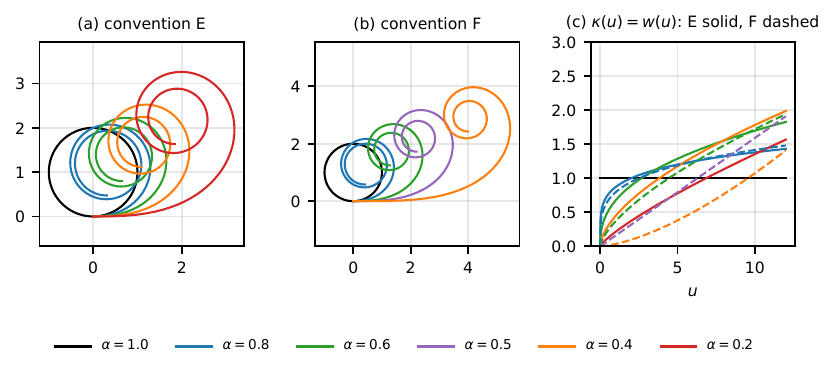}
\caption{Curves with $\kappa_\alpha\equiv1$, drawn up to total turning $4\pi$: convention E (left) and convention F (middle; $\alpha=1/2$ is the clothoid). Right: the corresponding Euclidean curvature $\kappa=w(u)$. Convention E is shown for $\alpha=1$, $0.8$, $0.6$, $0.4$, $0.2$ and convention F for $\alpha=1$, $0.8$, $0.6$, $0.5$, $0.4$; for $\alpha=0.2$ the curve of convention F is not shown, because its tangent angle grows like $u^5$ and reaches $4\pi$ only far outside the plotted region.}\label{fig:cesaro}
\end{figure}

\section{Fractional total turning for cornered curves}\label{sec:turning}
Let $\mathcal C$ be a closed, simple, positively oriented curve that is piecewise $C^2$ with finitely many corners, and let $\vartheta(u)$ be its right-continuous tangent angle. Then $\vartheta$ has bounded variation on $[0,L]$ and its turning measure is $d\vartheta=\kappa\,du+\sum_i\eta_i\,\delta_{u_i}$, with exterior angles $\eta_i\in(-\pi,\pi)$ at the corners $u_i$ \cite{Sullivan2008}.

\begin{definition}[Fractional total turning]\label{def:theta}
\[
\Theta_\alpha=\int_{(0,L]}\frac{d\vartheta(u)}{w(u)},
\]
whenever the integral converges (possibly improperly at $u=0$). For $C^2$ curves, $\Theta_\alpha=\int_0^L\kappa_\alpha\,du$.
\end{definition}

\begin{remark}\label{rem:dP}
Since $\kappa_\alpha=d\vartheta/dP$ (Lemma~\ref{lem:frenet}), $\int_0^{P(L)}\kappa_\alpha\,dP=\int_0^L\kappa\,du=2\pi$ for every simple closed $C^2$ curve and every $\alpha$, a trivial fractional turning-tangent theorem. $\Theta_\alpha$ instead weights the turning by the memory factor $1/w$, which emphasizes turning near the base point.
\end{remark}

\begin{proposition}\label{prop:theta}
Let $\mathcal C$ be as above, without corner at the base point, and let $\kappa$ be bounded on some interval $(0,\epsilon]$.
(a) If $p<1$, $\Theta_\alpha$ is finite; this holds for every $\alpha\in(0,1]$ in convention E and for $\alpha\in(1/2,1]$ in convention F.
(b) If $\kappa$ is continuous at $0$ with $\kappa(0)\neq0$ and $p\ge1$ (convention F with $\alpha\le1/2$), then $\Theta_\alpha$ diverges to $\operatorname{sgn}(\kappa(0))\,\infty$.
(c) If $\kappa\equiv0$ on $(0,\epsilon)$ (base point inside a straight edge), $\Theta_\alpha$ is finite for every $p\ge0$.
(d) For $\alpha=1$, $\Theta_1=2\pi$.
(e) If $|\kappa(u)|\le Mu^{q}$ on $(0,\epsilon]$ for some $M>0$ and $q>p-1$, then $\Theta_\alpha$ is finite; (a) is the case $q=0$. In particular, if $\kappa$ is $C^2$ near the base point, $\kappa(0)=0$, and $\mathcal C$ is symmetric about the normal line at the base point (so that $\kappa$ is even in $u$), one can take $q=2$, so that $\Theta_\alpha$ is finite for $p<3$, i.e.\ for every $\alpha$ in convention E and for $\alpha>1/4$ in convention F.
\end{proposition}
\begin{proof}
Shrinking $\epsilon$ if necessary, we may assume that no corner lies in $(0,\epsilon]$. On $(\epsilon,L]$ we have $1/w\le1/w(\epsilon)$ and $|d\vartheta|((\epsilon,L])\le\operatorname{TV}(\vartheta)<\infty$, so that part is finite. On $(0,\epsilon]$, $|\int\kappa/w\,du|\le\|\kappa\|_\infty C^{-1}\int_0^\epsilon u^{-p}du=\|\kappa\|_\infty\epsilon^{1-p}/(C(1-p))$ if $p<1$, which gives (a). In (b), $\kappa$ has constant sign and $|\kappa|\ge|\kappa(0)|/2$ on some $(0,\epsilon']$, so the integral over $(0,\epsilon']$ has modulus at least $\frac{|\kappa(0)|}{2C}\int_0^{\epsilon'}u^{-p}du=\infty$, while the rest is finite. In (c) the integrand vanishes on $(0,\epsilon)$. (d) $w\equiv1$ and $\int d\vartheta=2\pi$ by the theorem of turning tangents \cite{Hopf1935,DoCarmo1976}. (e) On $(0,\epsilon]$, $|\int\kappa/w\,du|\le MC^{-1}\int_0^\epsilon u^{q-p}du=M\epsilon^{1+q-p}/(C(1+q-p))<\infty$ because $q-p>-1$. If $\kappa$ is $C^2$ and even with $\kappa(0)=0$, then $\kappa'(0)=0$, and Taylor's formula gives $|\kappa(u)|\le\frac12\max|\kappa''|\,u^2$.
\qed
\end{proof}

\begin{remark}\label{rem:corner}
If the base point is a corner with exterior angle $\eta\neq0$, its contribution must be attributed either to $u=0$, where $1/w(0^+)=\infty$ for $p>0$, or to $u=L$, where it equals $\eta/w(L)$. The choice is arbitrary, so $\Theta_\alpha$ is not well defined and corners are not admissible base points. For the cornered test curves $(4;1,1,1)$ and $(4;2,1,1)$ the point $\varphi=0$ is a vertex, and we use the edge midpoint $\varphi_0=\pi/4$.
\end{remark}

\begin{proposition}\label{prop:polygon}
For a regular $N$-gon with edge length $\ell$ and base point at an edge midpoint, the corners lie at $u_j=(2j+1)\ell/2$, $j=0,\dots,N-1$, each with exterior angle $2\pi/N$, and
\begin{align*}
\Theta_\alpha&=\frac{2\pi}{N}\sum_{j=0}^{N-1}\frac{1}{w(u_j)}=\frac{2\pi}{NC}\Bigl(\frac{\ell}{2}\Bigr)^{-p}\sum_{j=0}^{N-1}(2j+1)^{-p},\\
\Theta^E_\alpha&=\frac{2\pi}{N}\,\frac{\Gamma(2-\alpha)}{\alpha}\Bigl(\frac{\ell}{2}\Bigr)^{\alpha-1}\sum_{j=0}^{N-1}(2j+1)^{\alpha-1}\qquad\text{(convention E)}.
\end{align*}
If $p<1$ and the perimeter $L=N\ell$ is fixed, then $\Theta_\alpha\to2\pi L^{-p}/(C(1-p))$ as $N\to\infty$, which is the value for the circle of perimeter $L$, with error $O(N^{-(1-p)})$; in convention E the error is $O(N^{-\alpha})$.
\end{proposition}
\begin{proof}
The formula follows from Definition~\ref{def:theta}. Since $f(x)=(2x+1)^{-p}$ is decreasing, $\int_0^Nf\le\sum_{j<N}f(j)\le f(0)+\int_0^{N-1}f$, whence $\sum_{j<N}(2j+1)^{-p}=(2N)^{1-p}/(2(1-p))+O(1)$ for $p<1$. With $\ell=L/N$ this gives $\Theta_\alpha=2\pi L^{-p}/(C(1-p))+O(N^{p-1})$. For the circle, $\Theta_\alpha=\frac{1}{RC}\int_0^Lu^{-p}du=2\pi L^{-p}/(C(1-p))$ with $R=L/2\pi$.
\qed
\end{proof}
The rhombus $|x|+|y|=1$ is the Gielis curve $(4;1,1,1)$ ($N=4$, $\ell=\sqrt2$, $L=4\sqrt2$). In convention E, $\Theta^E_{1/2}=7.9546$, against $9.3648$ for the circle of equal perimeter; in convention F, $\Theta^F_{3/4}=6.9971$ against $7.4944$. Table~\ref{tab:polygon} shows the slow convergence of polygons to the circle (rates fitted for $N=2^{10},2^{12},2^{14},2^{16}$ agree with $N^{-(1-p)}$ to three digits; see also Figure~\ref{fig:polygons}), and Tables~\ref{tab:thetaE}--\ref{tab:thetaF} and Figure~\ref{fig:theta} compare $\Theta_\alpha$ across curves rescaled to the same perimeter, using $\Theta_\alpha\mapsto c^{-p}\Theta_\alpha$ under $x\mapsto cx$.

\begin{table}[H]\centering
\caption{$\Theta_\alpha$ of regular $N$-gons with perimeter $4\sqrt2$ (base point at an edge midpoint) and of the circle with the same perimeter (Proposition~\ref{prop:polygon}).}\label{tab:polygon}
\begin{tabular}{crrrrrr}
\toprule
conv. & $\alpha$ & $N=4$ & $N=16$ & $N=256$ & $N=4096$ & circle \\
\midrule
E & 0.3 & 12.3928 & 14.5869 & 16.9994 & 18.0497 & 18.8593 \\
E & 0.5 & 7.9546 & 8.6571 & 9.1878 & 9.3205 & 9.3648 \\
E & 0.8 & 6.1930 & 6.3135 & 6.3673 & 6.3732 & 6.3739 \\
\midrule
F & 0.6 & 11.1925 & 12.9735 & 14.8106 & 15.5398 & 16.0196 \\
F & 0.75 & 6.9971 & 7.2956 & 7.4631 & 7.4895 & 7.4944 \\
F & 0.9 & 6.1958 & 6.2511 & 6.2724 & 6.2742 & 6.2744 \\
\bottomrule
\end{tabular}
\end{table}

\begin{figure}[H]\centering
\includegraphics[width=\textwidth]{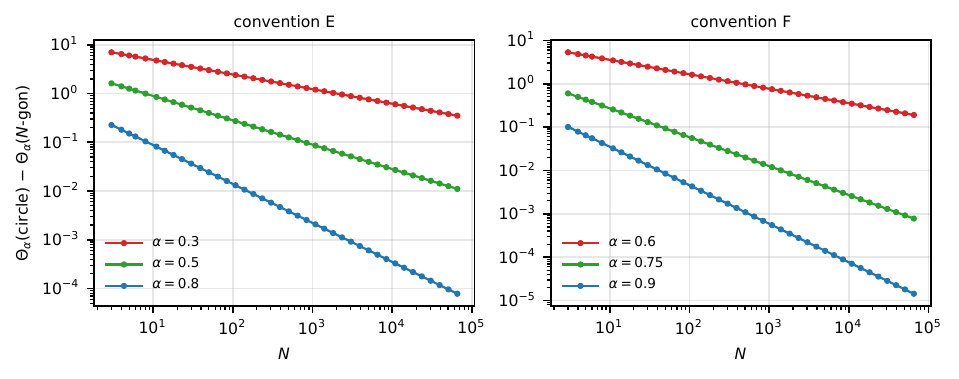}
\caption{Proposition~\ref{prop:polygon}: difference between $\Theta_\alpha$ of the circle and of the regular $N$-gon with the same perimeter $4\sqrt2$ (base point at an edge midpoint), convention E (left) and F (right). Dashed: slope $-(1-p)$, i.e.\ $N^{-\alpha}$ in convention E and $N^{-(2\alpha-1)/\alpha}$ in convention F.}\label{fig:polygons}
\end{figure}

\begin{table}[H]\centering
\caption{Fractional total turning in convention E for curves of common perimeter $L=4\sqrt2$. Base points: $\varphi_0=0$ for the smooth Gielis curves (a flat point of the squircle, $\kappa=0$; concave points of $(6;6,4,4)$ and $(5;7,7,7)$, with $\kappa=-1/2$ and $-9/16$ for $a=b=1$, i.e.\ $-0.603$ and $-0.760$ after rescaling to $L=4\sqrt2$), and the edge midpoint $\varphi_0=\pi/4$ for the cornered curves $(4;1,1,1)$ and $(4;2,1,1)$ (Remark~\ref{rem:corner}), with $\kappa(\varphi_0)=0$ and $2^{-3/4}$ for $a=b=1$, i.e.\ $0$ and $0.608$ after rescaling. All values equal $2\pi$ at $\alpha=1$.}\label{tab:thetaE}
\begin{tabular}{lrrrrr}
\toprule
curve (convention E) & $\alpha=0.2$ & $\alpha=0.4$ & $\alpha=0.6$ & $\alpha=0.8$ & $\alpha=1.0$ \\
\midrule
circle & 36.575 & 12.406 & 7.743 & 6.374 & 6.283 \\
rhombus $(4;1,1,1)$ & 18.359 & 9.543 & 7.014 & 6.193 & 6.283 \\
$(4;2,1,1)$ & 28.172 & 11.065 & 7.398 & 6.287 & 6.283 \\
squircle $(4;4,4,4)$ & 19.340 & 9.824 & 7.112 & 6.223 & 6.283 \\
$(6;6,4,4)$ & 12.020 & 9.009 & 6.977 & 6.204 & 6.283 \\
$(5;7,7,7)$ & 8.009 & 8.209 & 6.741 & 6.140 & 6.283 \\
\bottomrule
\end{tabular}
\end{table}

\begin{table}[H]\centering
\caption{As Table~\ref{tab:thetaE}, in convention F, for $\alpha\in(1/2,1]$, where $\Theta^F_\alpha$ is finite for all six curves (Proposition~\ref{prop:theta}).}\label{tab:thetaF}
\begin{tabular}{lrrrrr}
\toprule
curve (convention F) & $\alpha=0.6$ & $\alpha=0.7$ & $\alpha=0.8$ & $\alpha=0.9$ & $\alpha=1.0$ \\
\midrule
circle & 16.020 & 8.695 & 6.823 & 6.274 & 6.283 \\
rhombus $(4;1,1,1)$ & 11.193 & 7.754 & 6.548 & 6.196 & 6.283 \\
$(4;2,1,1)$ & 13.770 & 8.249 & 6.692 & 6.237 & 6.283 \\
squircle $(4;4,4,4)$ & 11.599 & 7.876 & 6.592 & 6.210 & 6.283 \\
$(6;6,4,4)$ & 10.048 & 7.688 & 6.558 & 6.204 & 6.283 \\
$(5;7,7,7)$ & 8.785 & 7.389 & 6.462 & 6.175 & 6.283 \\
\bottomrule
\end{tabular}
\end{table}

\begin{figure}[H]\centering
\includegraphics[width=\textwidth]{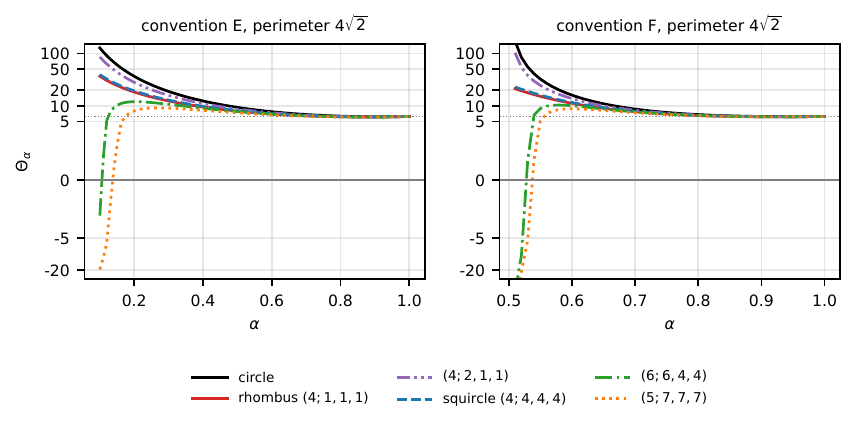}
\caption{$\Theta_\alpha$ versus $\alpha$ at perimeter $4\sqrt2$, convention E (left) and F (right), symmetric-logarithmic scale (linear for $|\Theta_\alpha|\le5$); dotted: $2\pi$. In convention F the circle and $(4;2,1,1)$ diverge to $+\infty$ and $(6;6,4,4)$, $(5;7,7,7)$ diverge to $-\infty$ as $\alpha\downarrow1/2$ (their curvature at the base point does not vanish), whereas the squircle ($\kappa(\varphi_0)=0$, $\kappa=O(u^2)$) and the rhombus (base point on an edge) stay finite, as predicted by Proposition~\ref{prop:theta}(e) and~(c), respectively.}\label{fig:theta}
\end{figure}

The Gielis curve $(4;2,1,1)$ has the same four corners as the rhombus ($n_2=n_3=1$, Lemma~\ref{lem:regular}), but curved edges: for $0\le\varphi\le\pi/2$, $r(\varphi)=2^{-1/4}\cos(\varphi-\pi/4)^{-1/2}$, so that $r'/r=\frac12\tan(\varphi-\pi/4)$. Since the tangent makes the angle $\operatorname{arccot}(r'/r)\in(0,\pi)$ with the radial direction, the exterior angles are $\eta=\pi-2\arctan2=\arctan\frac43$, and each edge turns by $\frac\pi2-\eta=\arctan\frac34$. At the base point $\varphi_0=\pi/4$ (an edge midpoint, Remark~\ref{rem:corner}) the curvature is $\kappa(\varphi_0)=2^{-3/4}>0$. Hence, unlike for the rhombus (Proposition~\ref{prop:theta}(c)), $\Theta^F_\alpha$ of $(4;2,1,1)$ is finite for $\alpha>1/2$ and diverges for $\alpha\le1/2$ (Proposition~\ref{prop:theta}(a), (b)). Its values are computed by quadrature of $\kappa/w$ on the four edges (five pieces between the base point and the corners), plus the corner terms $\eta/w(u_j)$ (Section~\ref{sec:numerics}). Figure~\ref{fig:turning} shows the turning measure and the partial fractional turning up to each arc length for the rhombus, $(4;2,1,1)$, the circle, the squircle and $(6;6,4,4)$.

\begin{figure}[H]\centering
\includegraphics[width=\textwidth]{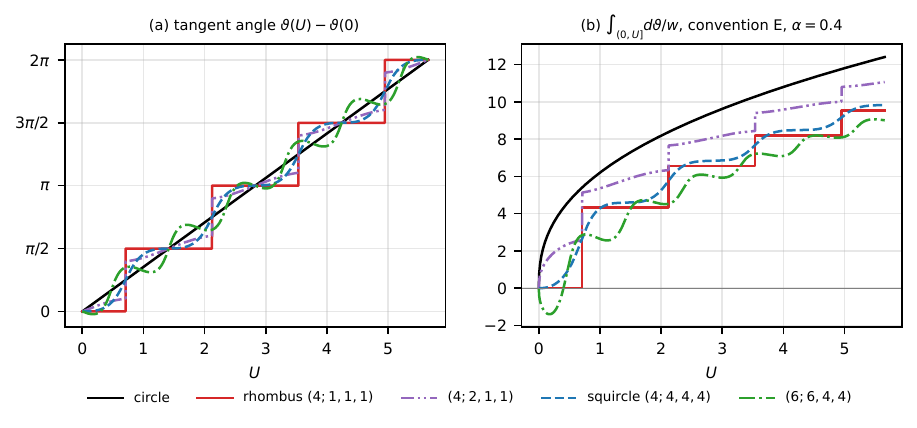}
\caption{Turning measure and fractional total turning (Definition~\ref{def:theta}, Proposition~\ref{prop:theta}) for curves of perimeter $4\sqrt2$, with the base points of Table~\ref{tab:thetaE}. (a) The tangent angle $\vartheta(U)-\vartheta(0)$ against the arc length $U$; corners are jumps ($\pi/2$ for the rhombus, $\arctan(4/3)$ for $(4;2,1,1)$). (b) The partial fractional turning $\int_{(0,U]}d\vartheta/w$ in convention E for $\alpha=0.4$; at $U=4\sqrt2$ it equals the entries of Table~\ref{tab:thetaE}. Near the base point the weight $1/w\propto u^{\alpha-1}$ is large: for $(6;6,4,4)$, whose base point is concave, the partial turning is negative at first.}\label{fig:turning}
\end{figure}

The fractional total turning $\Theta_\alpha$ need not be positive: it is dominated by the turning near the base point when $\alpha$ is small, and for $(6;6,4,4)$ and $(5;7,7,7)$, whose base points lie in concave parts of the curve, $\Theta^E_\alpha<0$ for $\alpha<0.106$ and $\alpha<0.137$, respectively, and $\Theta^F_\alpha<0$ for $1/2<\alpha<0.528$ and $1/2<\alpha<0.537$; for $\alpha\le1/2$, $\Theta^F_\alpha$ diverges to $-\infty$ (Proposition~\ref{prop:theta}(b)). In every case we computed, the circle has the largest $\Theta_\alpha$ for $\alpha<1$, in both conventions, and the cornered curve $(4;2,1,1)$ the second largest, above the smooth squircle, whose curvature vanishes at the base point. We do not claim this in general: $\Theta_\alpha$ depends on the base point, and a proof or counterexample would be an isoperimetric-type result for $\Theta_\alpha$, which we state as an open problem.

\section{Fractional helices over Gielis bases}\label{sec:helix}
Let $\mathcal B$ be a convex planar Gielis curve with arc length $\varsigma$ and curvature $\kappa_{\mathcal B}(\varsigma)\ge0$ (for example the squircle; Remark~\ref{rem:convexbase} gives convex Gielis curves for every $m$). For $\beta\in(0,\pi/2)$ the cylindrical helix $\gamma(u)=\bigl(\mathcal B(u\cos\beta),\,u\sin\beta\bigr)$ has arc length $u$ and, where $\kappa_{\mathcal B}>0$,
\[
\kappa=\kappa_{\mathcal B}(u\cos\beta)\cos^2\beta,\qquad \tau=\kappa_{\mathcal B}(u\cos\beta)\sin\beta\cos\beta,\qquad \tau/\kappa=\tan\beta .
\]

\begin{proposition}\label{prop:helix}
Let $w$ be either weight, $P=\int w$, $\kappa_\alpha=\kappa/w$, $\tau_\alpha=\tau/w$, and $'=d/dP$ as in Lemma~\ref{lem:frenet}.
(a) $\tau_\alpha/\kappa_\alpha=\tau/\kappa$. Hence the fractional general-helix condition $\kappa_\alpha/\tau_\alpha=\mathrm{const}$ \cite[Theorem~3.2]{HasYilmaz2023} characterizes exactly the classical general helices (Lancret's theorem \cite{Struik1988}).
(b) The fractional Izumiya--Takeuchi function satisfies
\[
\sigma_\alpha:=\frac{\kappa_\alpha^2}{(\kappa_\alpha^2+\tau_\alpha^2)^{3/2}}\Bigl(\frac{\tau_\alpha}{\kappa_\alpha}\Bigr)'
=\sigma:=\frac{\kappa^2}{(\kappa^2+\tau^2)^{3/2}}\,\frac{d}{du}\Bigl(\frac{\tau}{\kappa}\Bigr).
\]
Hence the fractional slant helices ($\sigma_\alpha$ constant) are exactly the classical slant helices of Izumiya and Takeuchi \cite{Izumiya2004}.
(c) If the derivative in (b) is taken with respect to $u$ while $\kappa_\alpha,\tau_\alpha$ are retained, one obtains $w\,\sigma$, whose constancy is not equivalent to that of $\sigma$ when $\alpha<1$.
\end{proposition}
\begin{proof}
(a) The factor $1/w$ cancels. (b) $\tau_\alpha/\kappa_\alpha=\tau/\kappa$, $d/dP=w^{-1}d/du$, and $\kappa_\alpha^2(\kappa_\alpha^2+\tau_\alpha^2)^{-3/2}=w\,\kappa^2(\kappa^2+\tau^2)^{-3/2}$. (c) The same computation without the factor $w^{-1}$.
\qed
\end{proof}

\begin{remark}\label{rem:HY34}
The condition obtained in \cite[Theorem~3.4]{HasYilmaz2023}, $(\tau_\alpha/\kappa_\alpha)'/\kappa_\alpha=\mathrm{const}$, is not the Izumiya--Takeuchi condition, even for $\alpha=1$: its derivation combines $\langle\bn_\alpha,U\rangle=\cos\psi$ with $\langle\bb_\alpha,U\rangle=\sin\psi$ (constant angle $\psi$, fixed unit vector $U$), which forces $\langle\bt_\alpha,U\rangle=0$ because $\langle\bt_\alpha,U\rangle^2+\langle\bn_\alpha,U\rangle^2+\langle\bb_\alpha,U\rangle^2=|U|^2=1$. The standard argument of \cite{Izumiya2004}, applied to \eqref{eq:frenetP}, gives $\sigma_\alpha=\pm\cot\psi$, i.e.\ Proposition~\ref{prop:helix}(b). A non-invariance of slant helices under Caputo scaling can only be obtained with mixed-parameter conditions such as (c); it is an artefact of differentiating with respect to a parameter other than the one in the fractional Frenet equations.
\end{remark}

Figure~\ref{fig:helix} shows a Gielis helix over the squircle: $\kappa_\alpha$ and $\tau_\alpha$ decay along the curve by the common factor $1/w$, while their ratio stays equal to $\tan\beta$.

\begin{figure}[H]\centering
\includegraphics[width=\textwidth]{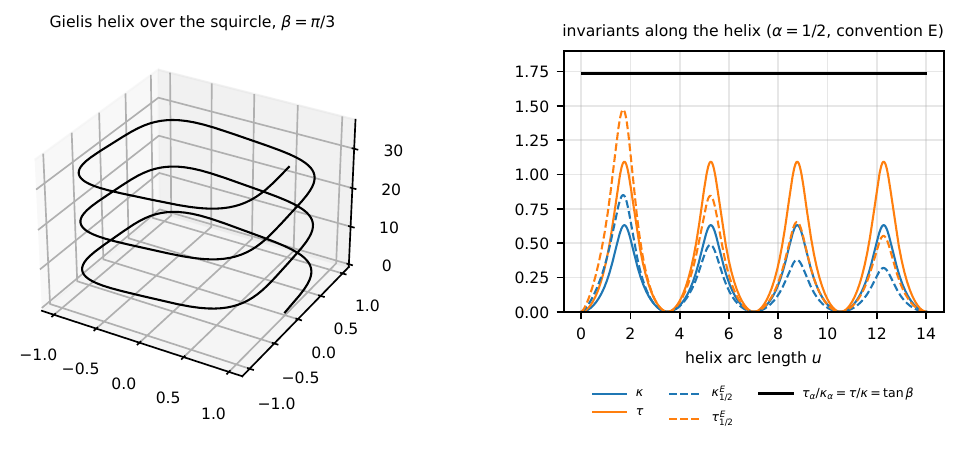}
\caption{Gielis helix over the squircle, $\beta=\pi/3$. Left: three turns; the axes are not to scale (the vertical axis is compressed). Right: $\kappa$, $\tau$ and their convention-E fractional versions for $\alpha=1/2$; the ratio $\tau_\alpha/\kappa_\alpha=\tau/\kappa=\tan\beta$ is constant (Proposition~\ref{prop:helix}(a)); it is undefined at the flat points of the base, where $\kappa=0$.}\label{fig:helix}
\end{figure}

\begin{remark}[Convex Gielis bases]\label{rem:convexbase}
The construction above needs a convex base. A curve $r=1/\rho(\varphi)$ is convex if and only if $\rho+\rho''\ge0$, because the numerator of~\eqref{eq:kappa} is $r^2+2r'^2-rr''=r^3(\rho+\rho'')$. For $(m;n_1,4,4)$ with $a=b=1$ we have $\rho=A^{1/n_1}$, $A=1-\tfrac12\sin^2(m\varphi/2)$, and a direct computation gives
\[
\rho+\rho''=\rho\,\Bigl[1-\frac{m^2}{n_1}\,\frac{3c+1}{(3+c)^2}+\frac{m^2}{16\,n_1^2}\Bigl(\frac{d\ln A}{dx}\Bigr)^2\Bigr],\qquad c=\cos(m\varphi),\quad x=\frac{m\varphi}{4}.
\]
Since $(3c+1)/(3+c)^2$ increases from $-1/2$ to $1/4$ on $[-1,1]$, the curve is convex if and only if $n_1\ge m^2/4$, and for $n_1=m^2/4$ its curvature vanishes exactly at the $m$ points $\varphi=2k\pi/m$. For $m=4$ this is the squircle; Figure~\ref{fig:helix_m} shows the helices over $(5;6.25,4,4)$, $(6;9,4,4)$ and $(7;12.25,4,4)$. The curves $(m;4,4,4)$ with $m\ge5$ are not convex ($\kappa(0)=1-m^2/16<0$), and neither are the test curves $(6;6,4,4)$ and $(5;7,7,7)$.
\end{remark}

\begin{figure}[H]\centering
\includegraphics[width=\textwidth]{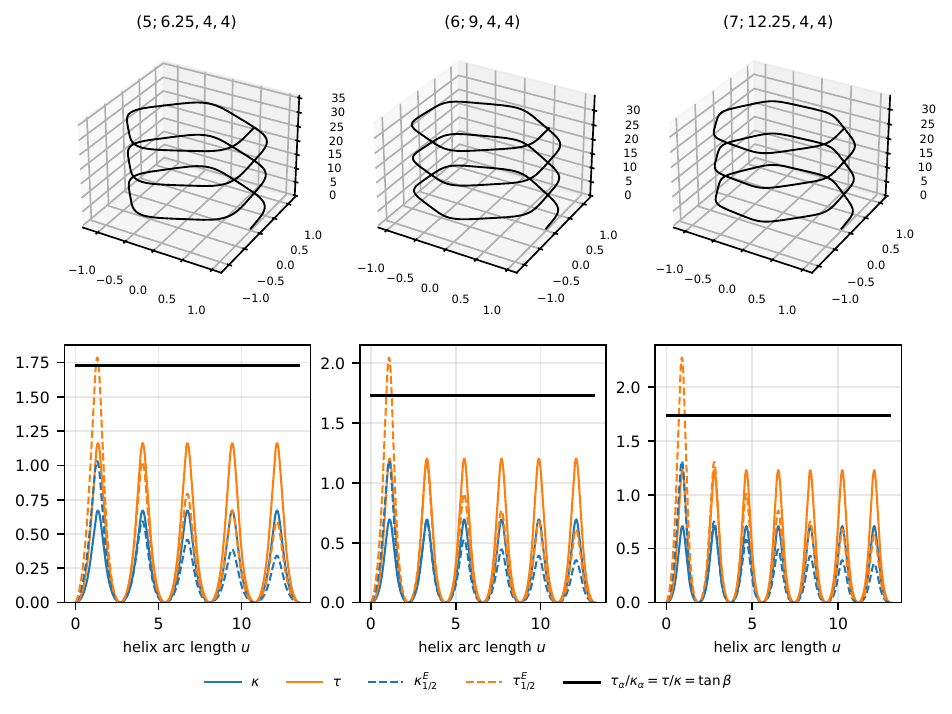}
\caption{Gielis helices with $\beta=\pi/3$ over the convex bases $(m;m^2/4,4,4)$ of Remark~\ref{rem:convexbase} for $m=5$, $6$, $7$ (the case $m=4$, the squircle, is Figure~\ref{fig:helix}). Top: three turns; the vertical axis is compressed. Bottom: $\kappa$, $\tau$ and their convention-E fractional versions for $\alpha=1/2$ along one turn; the ratio $\tau_\alpha/\kappa_\alpha=\tau/\kappa=\tan\beta$ is constant (Proposition~\ref{prop:helix}(a)); it is undefined at the $m$ flat points of the base per turn, where $\kappa=0$.}\label{fig:helix_m}
\end{figure}

\section{Computational protocol}\label{sec:numerics}
The library in \ref{app:code} first checks closure ($r$ must be $2\pi$-periodic, Lemma~\ref{lem:periodic}) and then computes $r,r',r''$ by spectral differentiation on a periodic grid of $N_\varphi=11520$ points, which is legitimate for the smooth test curves ($r\in C^6$, Remark~\ref{rem:n1}) and must not be used for cornered curves. From these it obtains $\kappa$ and $u$ by \eqref{eq:kappa}, and $\kappa_\alpha$ by \eqref{eq:scaling} in both conventions, leaving $\kappa_\alpha$ undefined at $u=0$ for $\alpha<1$. It evaluates $\Theta_\alpha$ through the substitution $W=u^{1-p}$, which removes the endpoint singularity, $\int\kappa\,u^{-p}du=(1-p)^{-1}\int\kappa\,dW$; the case $p\ge1$ is rejected, because there $\Theta_\alpha$ diverges if $\kappa(0)\neq0$ (Proposition~\ref{prop:theta}(b)) and Proposition~\ref{prop:theta}(e) gives only a sufficient condition for finiteness. Finally, it provides closed forms for circles and polygons.
The round-off error of spectral differentiation grows like $N_\varphi^k$ for the $k$-th derivative: for $(6;6,4,4)$ the error of $r''$ is of order $10^{-9}$ at $N_\varphi=4096$ but exceeds $10^{-6}$ at $N_\varphi=262144$, so that very fine grids are not suitable as references. Reference values are therefore computed independently, with a separate code based on the closed-form derivatives of \eqref{eq:gielis}, composite Gauss--Legendre quadrature for $u$, and adaptive quadrature with an algebraic weight at the base point.
Table~\ref{tab:convergence} shows the accuracy of $\Theta_\alpha$ under grid refinement: the relative error decreases like $N_\varphi^{-2}$ and is at most $1.5\times10^{-5}$ for $N_\varphi=4096$. Figure~\ref{fig:accuracy} shows this decay and the growth of the round-off in $r''$. At $N_\varphi=11520$, $\Theta_1$ reproduces $2\pi$ to within a relative error of $10^{-6}$ for the three smooth curves of Tables~\ref{tab:thetaE}--\ref{tab:thetaF} (and $4\times10^{-5}$ for $(8;4,6,6)$), and the circle's closed form to within a relative error of $10^{-9}$; the smooth-curve entries of Tables~\ref{tab:thetaE}--\ref{tab:thetaF} agree with the reference values to the digits shown (deviation below $3\times10^{-5}$).
The cornered curve $(4;2,1,1)$ is not smooth, so the spectral library is not used for it. Its entries in Tables~\ref{tab:thetaE}--\ref{tab:thetaF} and Figure~\ref{fig:theta} are computed with the same reference code, by adaptive quadrature of $\kappa\nu u^{-p}$ on the arcs between the corners (closed-form derivatives on each arc, algebraic weight at the base point), plus the corner terms $\eta_ju_j^{-p}$, with $\eta_j$ obtained from the one-sided tangents. This code reproduces the closed form of Proposition~\ref{prop:polygon} for the rhombus to within $10^{-12}$, yields $\eta_j=\arctan\frac43$ and $\Theta_1=2\pi$ to within $10^{-12}$, and agrees to within $10^{-10}$ with a second, independent implementation (\texttt{mpmath} quadrature, explicit form of $r$, exact $\eta$).
Finally, the quantitative content of the numbered statements of Sections~\ref{sec:prelim}--\ref{sec:helix} and every number quoted in the text are verified numerically, in 53 separate checks.

\begin{table}[H]\centering
\caption{Relative error of $\Theta_\alpha$ for $(6;6,4,4)$ with $a=b=1$ (not rescaled; $L=6.8272$) on grids of $N_\varphi$ points; reference values computed independently (Section~\ref{sec:numerics}); last row: $\alpha=1$ against the exact value $2\pi$.}\label{tab:convergence}
\begin{tabular}{cr rrrr}
\toprule
conv. & $\alpha$ & reference value & $N_\varphi=1024$ & $N_\varphi=4096$ & $N_\varphi=16384$ \\
\midrule
E & 0.2 & 10.341342 & 2.5e-04 & 1.5e-05 & 9.0e-07 \\
E & 0.5 & 7.091430 & 1.7e-05 & 1.0e-06 & 6.3e-08 \\
F & 0.6 & 8.863837 & 6.5e-05 & 3.9e-06 & 2.4e-07 \\
F & 0.8 & 6.256566 & 3.6e-06 & 2.3e-07 & 1.4e-08 \\
-- & 1.0 & $2\pi$ & 9.3e-06 & 5.8e-07 & 3.6e-08 \\
\bottomrule
\end{tabular}
\end{table}

\begin{figure}[H]\centering
\includegraphics[width=\textwidth]{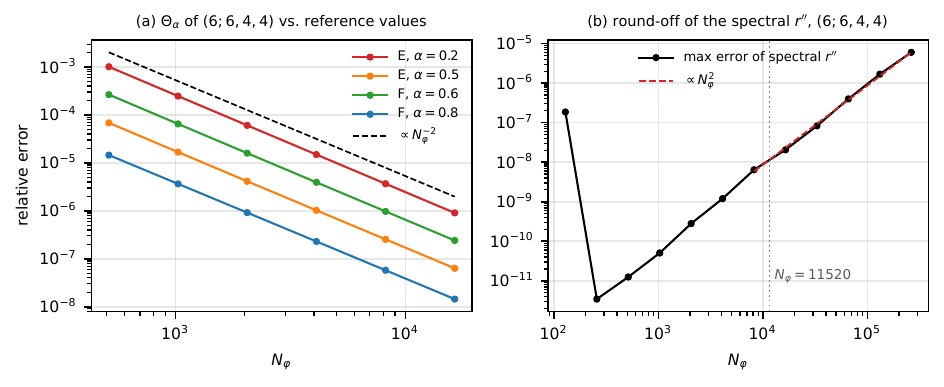}
\caption{Numerical accuracy for $(6;6,4,4)$ (Section~\ref{sec:numerics}). (a) Relative error of $\Theta_\alpha$ against the reference values of Table~\ref{tab:convergence}; dashed: slope $-2$. (b) Maximal error of the spectral second derivative $r''$ against its closed form: it is smallest on a coarse grid (here $N_\varphi=256$) and then grows like $N_\varphi^2$ through round-off (dashed); dotted: $N_\varphi=11520$, the grid used for Tables~\ref{tab:thetaE}--\ref{tab:thetaF}.}\label{fig:accuracy}
\end{figure}

\section{Discussion and conclusions}\label{sec:discussion}
The smoothness of Gielis curves is governed by $n_2,n_3$, not by $n_1$ (Lemma~\ref{lem:regular}). A ``regularizing'' change of metric such as the Gielis chart makes every Gielis curve that closes after one turn (Lemma~\ref{lem:periodic}) a unit circle, but the non-smoothness of $r$ reappears in the regularity of the metric (Proposition~\ref{prop:chart}, Corollary~\ref{cor:transport}); the superelliptic star derivative is not tangent to the transformed curve (Remark~\ref{rem:superelliptic}).

The Caputo scaling exists in two conventions, E and F, and every fractional quantity depends on the convention and on the base point (Lemma~\ref{lem:param}, Remarks~\ref{rem:HY} and~\ref{rem:base}). Under homotheties the fractional curvature has scaling exponent $2-\alpha$ in convention E and $1/\alpha$ in convention F (Theorem~\ref{thm:scaling}), and the Gielis parameters $(a,b)$ act as a homothety only when $n_2=n_3$ (Corollary~\ref{cor:param}). Caputo memory breaks the rotational symmetry (Theorem~\ref{thm:memory}), and the curves of constant fractional curvature have $\vartheta=\vartheta(0)+\kappa_0P(u)$, the clothoid being the case $\alpha=1/2$ of convention F (Theorem~\ref{thm:cesaro}).

The fractional total turning (Definition~\ref{def:theta}, Remark~\ref{rem:dP}) is finite on cornered curves under the conditions of Proposition~\ref{prop:theta}, has closed forms for polygons, converges from polygons to the circle at the rate $N^{-(1-p)}$ (Proposition~\ref{prop:polygon}), and may be negative. General helices and slant helices are Caputo-invariant when derivatives are taken with respect to the fractional parameter (Proposition~\ref{prop:helix}, Remark~\ref{rem:HY34}).

All results of this paper concern the simplified fractional tangent \eqref{eq:yajima} \cite{Yajima2018,Aydin2021,HasYilmaz2023}. The exact Caputo derivative of a curve is in general not tangent to it \cite{RubioLopez2023,RubioLopez2024}: for the unit circle parametrized by arc length, ${}^CD^{1/2}_ux(u)$ makes angles of $9.5^\circ$, $18.9^\circ$ and $36.8^\circ$ with $\bt(u)$ at $u=0.5$, $1$ and $2$. The invariance of the Frenet frame is therefore a property of the simplification, not of the Caputo operator. The results also depend on the base point and on the convention (E or F). The smooth Gielis curves in the numerical results have $n_2,n_3\ge4$; cornered curves are covered by closed forms for polygons, the Gielis rhombus $(4;1,1,1)$ included, and numerically by one curved-edge example, $(4;2,1,1)$; curves with $1<n_2<2$ or $1<n_3<2$ ($C^1$, with unbounded curvature; Corollary~\ref{cor:geom}) and cusps ($n_2<1$ or $n_3<1$) are not treated.

Several questions remain open: a periodic Caputo operator compatible with the dihedral symmetry, isoperimetric inequalities for $\Theta_\alpha$, fractional Gauss and mean curvatures of Gielis surfaces, a Frenet theory of Gielis curves based on the exact Caputo derivative in the spirit of \cite{RubioLopez2023,RubioLopez2024,Ogrenmis2025,RubioLopez2026}, and the calibration of $\alpha$ as a memory parameter in growth models.

\section*{Declarations}
\noindent\textbf{Code availability.} The Python library used in the computations is listed in full in \ref{app:code}.

\smallskip
\noindent\textbf{Funding.} This research received no specific grant from any funding agency in the public, commercial or not-for-profit sectors.

\smallskip
\noindent\textbf{Competing interests.} The authors declare that they have no known competing financial interests or personal relationships that could have appeared to influence the work reported in this paper.

\smallskip
\noindent\textbf{Author contributions.} Both authors contributed equally to this work and approved the final manuscript.

\appendix
\section{Source code of the library (Python 3, NumPy/SciPy)}\label{app:code}
The Python library used in the computations of this paper is listed below.
\begin{lstlisting}
import numpy as np
from scipy.special import gamma
from scipy.integrate import trapezoid, cumulative_trapezoid

CONVENTIONS = ("E", "F")


# ----------------------------------------------------------------------------
# Gielis curves
# ----------------------------------------------------------------------------
def gielis_r(phi, m, n1, n2, n3, a=1.0, b=1.0):
    return (np.abs(np.cos(m * phi / 4) / a) ** n2 +
            np.abs(np.sin(m * phi / 4) / b) ** n3) ** (-1.0 / n1)


def is_2pi_periodic(m, n1, n2, n3, a=1.0, b=1.0, tol=1e-12, n_test=4001):
    phi = np.linspace(0.0, 2 * np.pi, n_test)
    r0 = gielis_r(phi, m, n1, n2, n3, a, b)
    r1 = gielis_r(phi + 2 * np.pi, m, n1, n2, n3, a, b)
    return bool(np.max(np.abs(r1 - r0)) <= tol * np.max(r0))


def rotation_order(m, n1, n2, n3, a=1.0, b=1.0):
    if a == b and n2 == n3:
        return int(m)
    if int(m) % 2 == 0:
        return int(m) // 2
    raise ValueError("odd m with a != b or n2 != n3: r is not 2*pi-periodic")


def fft_deriv(f, order, period=2 * np.pi):
    n = len(f)
    k = np.fft.fftfreq(n, d=period / n) * 2 * np.pi
    return np.real(np.fft.ifft((1j * k) ** order * np.fft.fft(f)))


def gielis_geometry(m, n1, n2, n3, a=1.0, b=1.0, N=11520, phi0=0.0):
    if not is_2pi_periodic(m, n1, n2, n3, a, b):
        raise ValueError("r is not 2*pi-periodic for these parameters (Lemma 1)")
    if min(n2, n3) < 2:
        raise ValueError("n2 or n3 < 2: r is not C^2 (Corollary 1)")
    phi = phi0 + np.linspace(0.0, 2 * np.pi, N, endpoint=False)
    r = gielis_r(phi, m, n1, n2, n3, a, b)
    r1 = fft_deriv(r, 1)
    r2 = fft_deriv(r, 2)
    nu = np.sqrt(r ** 2 + r1 ** 2)
    kappa = (r ** 2 + 2 * r1 ** 2 - r * r2) / nu ** 3
    dphi = 2 * np.pi / N
    u = np.concatenate([[0.0], np.cumsum(0.5 * (nu[:-1] + nu[1:]) * dphi)])
    L = u[-1] + 0.5 * (nu[-1] + nu[0]) * dphi
    return dict(phi=phi, r=r, x=r * np.cos(phi), y=r * np.sin(phi),
                kappa=kappa, u=u, L=L, nu=nu)


# ----------------------------------------------------------------------------
# Caputo weights (Section 2.2)
# ----------------------------------------------------------------------------
def weight_params(alpha, conv="E"):
    if not (0.0 < alpha <= 1.0):
        raise ValueError("alpha must lie in (0, 1]")
    if conv == "E":
        return alpha / gamma(2.0 - alpha), 1.0 - alpha
    if conv == "F":
        return (alpha ** (2.0 / alpha - 1.0) * gamma(2.0 - alpha) ** (-1.0 / alpha),
                (1.0 - alpha) / alpha)
    raise ValueError("conv must be 'E' or 'F'")


def weight(u, alpha, conv="E"):
    C, p = weight_params(alpha, conv)
    return C * np.asarray(u, dtype=float) ** p


def frac_param(u, alpha, conv="E"):
    C, p = weight_params(alpha, conv)
    return C * np.asarray(u, dtype=float) ** (p + 1.0) / (p + 1.0)


def kappa_alpha(geo, alpha, conv="E"):
    C, p = weight_params(alpha, conv)
    k = np.asarray(geo["kappa"], dtype=float)
    u = np.asarray(geo["u"], dtype=float)
    if p == 0.0:
        return k.copy()
    out = np.full_like(k, np.nan)
    pos = u > 0
    out[pos] = k[pos] / (C * u[pos] ** p)
    return out


# ----------------------------------------------------------------------------
# Fractional total turning (Section 5)
# ----------------------------------------------------------------------------
def theta_frac(geo, alpha, conv="E"):
    C, p = weight_params(alpha, conv)
    if p >= 1.0:
        raise ValueError("p >= 1 is not supported: Theta diverges if kappa(0) != 0 "
                         "(Proposition 2(b)); see Proposition 2(e)")
    W = np.append(geo["u"], geo["L"]) ** (1.0 - p)
    kap = np.append(geo["kappa"], geo["kappa"][0])      # close the loop
    return trapezoid(kap, W) / (C * (1.0 - p))


def theta_rescaled(theta, alpha, conv, factor):
    _, p = weight_params(alpha, conv)
    return theta * factor ** (-p)


def theta_circle(alpha, L, conv="E"):
    C, p = weight_params(alpha, conv)
    if p >= 1.0:
        return np.inf
    return 2 * np.pi * L ** (-p) / (C * (1.0 - p))


def theta_polygon(alpha, n_sides, edge, conv="E"):
    C, p = weight_params(alpha, conv)
    j = np.arange(n_sides)
    uj = (2 * j + 1) * edge / 2.0
    return (2 * np.pi / n_sides) * np.sum(1.0 / (C * uj ** p))


def theta_rhombus(alpha, conv="E"):
    return theta_polygon(alpha, 4, np.sqrt(2.0), conv)


# ----------------------------------------------------------------------------
# Curves of constant fractional curvature (Theorem 3)
# ----------------------------------------------------------------------------
def const_curvature_curve(alpha, conv="E", kappa0=1.0, turning=4 * np.pi,
                          npts=20001):
    C, p = weight_params(alpha, conv)
    L = ((p + 1.0) * turning / (kappa0 * C)) ** (1.0 / (p + 1.0))
    u = np.linspace(0.0, L, npts)
    th = kappa0 * frac_param(u, alpha, conv)
    x = cumulative_trapezoid(np.cos(th), u, initial=0.0)
    y = cumulative_trapezoid(np.sin(th), u, initial=0.0)
    return u, x, y, kappa0 * weight(u, alpha, conv)


# ----------------------------------------------------------------------------
# Cylindrical helix over a convex planar base (Section 6)
# ----------------------------------------------------------------------------
def helix_over_base(geo, beta):
    u = geo["u"] / np.cos(beta)
    kap = geo["kappa"] * np.cos(beta) ** 2
    tau = geo["kappa"] * np.sin(beta) * np.cos(beta)
    return u, kap, tau
\end{lstlisting}


\end{document}